\documentclass[11pt,reqno]{amsart}
\usepackage{amssymb,amsmath,amsfonts,amsthm,enumerate,stmaryrd,tensor,mathtools,dsfont,upgreek,bbm,mathrsfs,nameref,microtype,xcolor,bm,tikz,stmaryrd}
\usetikzlibrary{arrows}
\usepackage[left=0.7 in, right=0.7 in,top=0.7 in, bottom=0.7 in]{geometry}
\usepackage{hyperref,cleveref}

\numberwithin{equation}{section}
\definecolor{popblue}{RGB}{55,115,255}
\definecolor{lightbl}{RGB}{155,205,255}
\definecolor{depthbl}{RGB}{145,215,255}
\definecolor{fancyre}{RGB}{225,55,115}
\definecolor{darkblu}{RGB}{15,75,185}
\definecolor{mellowy}{RGB}{225,225,35}
\renewcommand{\tilde}[1]{\widetilde{#1}}
\renewcommand{\Bar}{\overline}

\renewcommand{\S}{\mathbb{S}}
\newcommand{\R}{\mathbb{R}}
\newcommand{\N}{\mathbb{N}}
\newcommand{\Z}{\mathbb{Z}}

\newcommand{\C}{\mathbb{C}}

\newcommand{\T}{\mathbb{T}}

\newcommand{\imp}{\;\Rightarrow\;}
\newcommand{\m}{\mathrm}

\newcommand{\lv}{\lVert}
\newcommand{\rv}{\rVert}

\newcommand{\al}{\alpha}
\newcommand{\be}{\beta}
\newcommand{\es}{\varnothing}

\newcommand{\ep}{\varepsilon}
\newcommand{\f}{\frac}
\newcommand{\sig}{\sigma}
\newcommand{\gam}{\gamma}
\newcommand{\del}{\delta}

\newcommand{\pd}{\partial}

\newcommand{\grad}{\nabla}
\newcommand{\bpm}{\begin{pmatrix}}
	\newcommand{\epm}{\end{pmatrix}}

\renewcommand{\bar}{\overline}

\newcommand{\emb}{\hookrightarrow}

\renewcommand{\le}{\leqslant}

\newcommand{\bnorm}[1]{\Big\lv#1\Big\rv}

\newcommand{\tnorm}[1]{\lv#1\rv}
\newcommand{\bp}[1]{\Big(#1\Big)}
\renewcommand{\sp}[1]{\big(#1\big)}
\newcommand{\tp}[1]{(#1)}
\newcommand{\tabs}[1]{|#1|}
\newcommand{\tsb}[1]{[{#1}]}
\newcommand{\bcb}[1]{\Big\{{#1}\Big\}}
\newcommand{\tcb}[1]{\{{#1}\}}
\providecommand{\tbr}[1]{\langle #1 \rangle}
\renewcommand{\bf}[1]{\mathbf{#1}}
\newcommand{\ii}{\m{i}}
\newtheorem{prop}{\color{popblue}{Proposition}}[section]
\newtheorem{thm}[prop]{\color{popblue}{Theorem}}

\newtheorem{lem}[prop]{\color{popblue}{Lemma}}

\newtheorem{innercustomthm}{\color{popblue}{Theorem}}
\newenvironment{customthm}[1]
{\renewcommand\theinnercustomthm{#1}\innercustomthm}
{\endinnercustomthm}
\author{Noah Stevenson}
\address{
	Department of Mathematics\\
	Princeton University\\
	Princeton, NJ 08544, USA
}
\email[N. Stevenson]{stevenson@princeton.edu}
\thanks{N. Stevenson was supported by an NSF Graduate Research Fellowship}
\title[Roll wave solutions to the inclined shallow water equations]{
	Periodic gravity-capillary roll wave solutions to the inclined viscous shallow water equations in two dimensions
}
\subjclass[2020]{Primary 35Q35, 35C07; Secondary 35B32, 76A20}

\keywords{roll waves, traveling waves, viscous shallow water equations, bifurcation theory}
\begin{document}
	% _+__+_ -_+__+_ -_+__+_ -_+__+_ -_+__+_ -_+__+_ -_+__+_ -_+__+_ -_+__+_ -_+__+_ -_+__+_ -_+__+_ -_+__+_ -
	\begin{abstract}
		We study periodic, two-dimensional, gravity-capillary traveling wave solutions to a viscous shallow water system posed on an inclined plane. While thinking of the Reynolds and Bond numbers as fixed and finite, we vary the speed of the traveling frame and the degree of the incline and identify a set of the latter two parameters that classifies from which combinations nontrivial and small amplitude solution curves originate. Our principal technical tools are a combination of the implicit function theorem and a local multiparameter bifurcation theorem. To the best of the author's knowledge, this paper constitutes the first construction and mathematical study of properly two dimensional examples of viscous roll waves.
	\end{abstract}
	%*_*_*()(())()*_*_*()(())()*_*_*()(())()*_*_*()(())()*_*_*()(())()*_*_*()(())()*_*_*()(())()*_*_*()(())()
	\maketitle
	%*_*_*()(())()*_*_*()(())()*_*_*()(())()*_*_*()(())()*_*_*()(())()*_*_*()(())()*_*_*()(())()*_*_*()(())()
	\section{Introduction}
	%*_*_*()(())()*_*_*()(())()*_*_*()(())()*_*_*()(())()*_*_*()(())()*_*_*()(())()*_*_*()(())()*_*_*()(())()
	
	The phenomenon of roll waves are a frequently observed manifestation of the instability of a shallow layer of water driven down an incline and are caused by the interaction of the gravitational force and the friction between the fluid and the bottom. While the mathematics and physics literature has been exploring these features for over a century now, the vast majority of rigorous nonlinear results on the existence, stability, and instability of roll waves have been exclusively established for one dimensional models such as the Saint-Venant equations with empirical modifications. The purpose of this work is to explore some of the two dimensional nature of roll waves. We carry out the first construction of families of properly two dimensional small amplitude gravity-capillary roll wave solutions to a damped viscous shallow water model representing a thin film of tilted incompressible fluid. Our solution families are also examples of traveling wave solutions to a dissipative fluid system with the notable novel feature of existing purely in the presence of a gravitational force (see the end of Section~\ref{subsection on survey of previous work} for further discussion).
	
	%*_*_*()(())()*_*_*()(())()*_*_*()(())()*_*_*()(())()*_*_*()(())()*_*_*()(())()*_*_*()(())()*_*_*()(())()
	\subsection{The tilted shallow water equations, traveling formulation}
	%*_*_*()(())()*_*_*()(())()*_*_*()(())()*_*_*()(())()*_*_*()(())()*_*_*()(())()*_*_*()(())()*_*_*()(())()
	
	The fluid model under consideration in this work is the damped shallow water equations, built with the effects of viscosity, gravity, and capillarity, posed on a two-dimensional inclined plane. These are the system:
	\begin{equation}\label{The time dependent inclined shallow water equation}
		\begin{cases}
			\Bar{\upeta}(\pd_t\Bar{\upupsilon}+ \Bar{\upupsilon} \cdot   \grad \Bar{\upupsilon})   
			+ \alpha\Bar{\upupsilon} - \upmu\grad\cdot(\Bar{\upeta} \mathbb{S}\Bar{\upupsilon}  ) +  \Bar{\upeta}\grad(g \Bar{\upeta} - \upsigma \Delta \Bar{\upeta}) - \upkappa e_1\Bar{\upeta} = 0,\\
			\pd_t \Bar{\upeta} + \grad\cdot(\Bar{\upeta}\Bar{\upupsilon}) =0.
		\end{cases}
	\end{equation}
	Here the unknowns are the velocity vector field $\Bar{\upupsilon}:\R^+\times\R^2\to\R^2$ and the free surface height $\Bar{\upeta}:\R^+\times\R^2\to\R^+$. The matrix field $\mathbb{S}\Bar{\upupsilon}:\R^+\times\R^2\to\R^{2\times 2}$, which is the shallow water version of the viscous stress tensor, is given by the formula
	\begin{equation}\label{viscous stress tensor}
		\mathbb{S}\Bar{\upupsilon}=\grad\Bar{\upupsilon}+\grad\Bar{\upupsilon}^{\m{t}}+2(\grad\cdot\Bar{\upupsilon})I_{2\times 2}.
	\end{equation}
	The physical parameters of system~\eqref{The time dependent inclined shallow water equation} are inverse slip coefficient $\al\in\R^+$, the surface tension coefficient $\upsigma\in\R^+$, the viscosity $\upmu\in\R^+$, and the gravitational acceleration's components $g\in\R^+$, $\upkappa\in\R$; $g$ indicates the strength of the portion of gravity which is `normal' to the fluid while $\upkappa e_1$ is the component of gravity acting parallel to the fluid (which we choose to be in the $e_1$ direction).
	
	The shallow water model~\eqref{The time dependent inclined shallow water equation}, which is also known as the Saint-Venant equations when $\al=\upmu=\upsigma=\upkappa=0$, is an important system of equations which approximates the incompressible free boundary Navier-Stokes in scenarios where the fluid depth is much smaller than its characteristic horizontal scale. Consequently, the model is both physically and practically relevant. One should therefore think of solutions to~\eqref{The time dependent inclined shallow water equation} as describing flow of a thin film of incompressible liquid down an inclined plane.
	
	A derivation of the system~\eqref{The time dependent inclined shallow water equation} from the free boundary Navier-Stokes system appended with gravitational forcing of the form $-g e_3+\upkappa e_1$ and Navier slip boundary conditions on the rigid bottom proceeds through an asymptotic expansion and vertical averaging procedure; for the full details in the case $\upkappa=0$ we refer, for instance, to the surveys of Bresch~\cite{MR2562163} or Mascia~\cite{mascia_2010} and, in the case of general applied forcing, to Appendix B.1 in Stevenson and Tice~\cite{stevenson2023shallow}. 
	
	System~\eqref{The time dependent inclined shallow water equation} evidently has a number of structural resemblances to the barotropic compressible Navier-Stokes system, with $\Bar{\upupsilon}$ and $\Bar{\upeta}$ analogous to the fluid velocity and density, respectively. The term $\Bar{\upeta}\grad\tp{g\Bar{\upeta}}=\f{g}{2}\grad\tp{\Bar{\upeta}^2}$ corresponds to a quadratic pressure law while $\upmu\Bar{\upeta}$ corresponds to viscosity coefficients which depend linearly on the density. Continuing with this analogy, we shall refer to the first equation in~\eqref{The time dependent inclined shallow water equation} as the momentum equation and the second one as the continuity equation. This latter equation dictates how the free surface is deformed and transported by the velocity field.
	
	The advective derivative $\Bar{\upeta}\tp{\pd_t\Bar{\upupsilon}+\Bar{\upupsilon}\cdot\grad\Bar{\upupsilon}}$ in the momentum equation of~\eqref{The time dependent inclined shallow water equation} is balanced by several notable terms. First, we have the shallow water limit's manifestation of the Navier slip boundary condition, which is the zeroth order frictional damping term $\al\Bar{\upupsilon}$. We note that there are other options for a friction term that appear in the literature; this effect appears to be commonly modeled with the choice of the Ch\'ezy drag term, which would have the form $\tilde{\al}|\Bar{\upupsilon}|\Bar{\upupsilon}$. This is an empirically derived force which is meant to capture so-called turbulent friction. We disregard this option of frictional damping and adopt the linear choice $\al\Bar{\upupsilon}$, as it is consistent with the aforementioned mathematical derivation of our shallow water model from the Navier-Stokes equations.
	
	Next, the momentum equation inherits a viscous damping term, $-\upmu\grad\cdot\tp{\Bar{\upeta}\mathbb{S}\Bar{\upupsilon}}$, which captures the dissipative effect of intra-fluid friction. The term $\Bar{\upeta}\grad\tp{g\Bar{\upeta}-\upsigma\Delta\Bar{\upeta}}$ encodes how changes in the geometry of the free surface influence the velocity. We note that this term can be expressed as the divergence of a gravity-capillary stress tensor, namely $\Bar{\upeta}\grad\tp{g\Bar{\upeta}-\sig\Delta\Bar{\upeta}}=\grad\cdot\mathbb{G}_{g,\upsigma}(\Bar{\upeta})$, where
	\begin{equation}\label{the gravity capilary stress tensor}
		\mathbb{G}_{g,\upsigma}(\Bar{\upeta})=\f{g}{2}\Bar{\upeta}^2I_{2\times 2}-\f{\upsigma}{2}\tp{\Delta\tp{\Bar{\upeta}^2}-|\grad\Bar{\upeta}|^2}I_{2\times 2}+\upsigma\grad\Bar{\upeta}\otimes\grad\Bar{\upeta}.
	\end{equation}
	Finally, the term $-\upkappa e_1\Bar{\eta}$ encodes that the shallow water system is posed on an inclined plane and the size of the coefficient $\upkappa$ dictates the steepness level, with $\upkappa=0$ corresponding to a flat plane.
	
	We are interested in the question of whether or not system~\eqref{The time dependent inclined shallow water equation} admits certain nontrivial solutions called roll waves - which are stationary in a frame traveling at a constant velocity down the incline. Let us begin our search for these by first identifying and reformulating perturbatively around the equilibrium solutions. One notes that for any choice of equilibrium height $H\in\R$ the system~\eqref{The time dependent inclined shallow water equation} admits a constant solution of the form $\Bar{\upeta}=H$ and $\Bar{\upupsilon}=\tp{\kappa H/\al}e_1$. Let us fix some choice $H\in\R^+$, as the nonpositive choices are not physically meaningful.
	
	We shall next perform a rescaling and perturbative reformulation in which $(\al,g,H)\mapsto(1,1,1)$. We set $\m{L}=H\sqrt{gH}/\al$ and $\m{T}=H/\al$ to be some characteristic length and time of the system. Using these we then define the following non-dimensional parameters: inverse Reynolds number $\mu=\upmu T/L^2$, inverse Bond number $\sig=\upsigma/gL^2$, and incline steepness (or inverse Froude number) $\kappa=\f{\upkappa H T}{\al L}$. We then define the non-dimensional velocity $\upupsilon:\R^+\times\R^2\to\R^2$ and  free surface $\upeta:\R^+\times\R^2\to\R$ via
	\begin{equation}
		\Bar{\upupsilon}(t,x)=\f{\m{L}}{\m{T}}\tp{\kappa e_1+\upupsilon(t/T,x/L)},\quad\Bar{\upeta}(t,x)=H(1+\upeta(t/T,x/L)).
	\end{equation}
	In terms of $\upupsilon$ and $\upeta$, system~\eqref{The time dependent inclined shallow water equation} takes the non-dimensional form
	\begin{equation}\label{non-dimensional time dependent equations}
		\begin{cases}
			(1+\upeta)\tp{\pd_t\upupsilon+(\upupsilon+\kappa e_1)\cdot\grad\upupsilon}+\upupsilon-\mu\grad\cdot((1+\upeta)\mathbb{S}\upupsilon)+(1+\upeta)\grad(1-\sig\Delta)\upeta-\kappa e_1\upeta=0,\\
			\pd_t\upeta+\kappa\pd_1\upeta+\grad\cdot((1+\upeta)\upupsilon)=0.
		\end{cases}
	\end{equation}
	Note that $(\upupsilon,\upeta)=(0,0)$ is a solution to these equations representing the previously mentioned trivial equilibrium. We are interested in nontrivial, traveling perturbations of this solution. We are thus lead to reformulate the equations yet again, but this time in a frame that is traveling in the $e_1$ direction (which is parallel to the tilt direction) at a signed non-dimensional speed $\Bar{\gam}\in\R$. The following traveling ansatz is made: there are $u:\R^2\to\R^2$ and $\eta:\R^2\to\R$ with the property that 
	\begin{equation}\label{traveling ansatz}
		\upupsilon(t,x)=u(x-t\Bar{\gam} e_1),\quad\upeta(t,x)=\eta(x-t\Bar{\gam} e_1).
	\end{equation}
	Let us define the parameter $\gam=\Bar{\gam}-\kappa\in\R$ to represent the relative non-dimensional wave speed. From~\eqref{non-dimensional time dependent equations} and~\eqref{traveling ansatz} we then derive the equations satisfied by $u$, $\eta$, $\gam$, and $\kappa$:
	\begin{equation}\label{traveling form of the equations}
		\begin{cases}
			(1+\eta)(u-\gam e_1)\cdot\grad u+u-\mu\grad\cdot((1+\eta)\mathbb{S}u)+(1+\eta)\grad(1-\sig\Delta)\eta-\kappa e_1\eta=0,\\
			-\gam\pd_1\eta+\grad\cdot((1+\eta)u)=0.
		\end{cases}
	\end{equation}
	This is the form of the tilted shallow water system which we shall analyze in this work. We view the above system as a sort of nonlinear `eigenvalue' problem: Thinking of $\mu,\sig>0$ as being fixed, we wish to find $(\gam,\kappa)\in\R^2$ and $(u,\eta)\neq 0$ such that~\eqref{traveling form of the equations} is satisfied.
	
	\begin{figure}[!h]
		\centering
		\scalebox{0.45}{\includegraphics{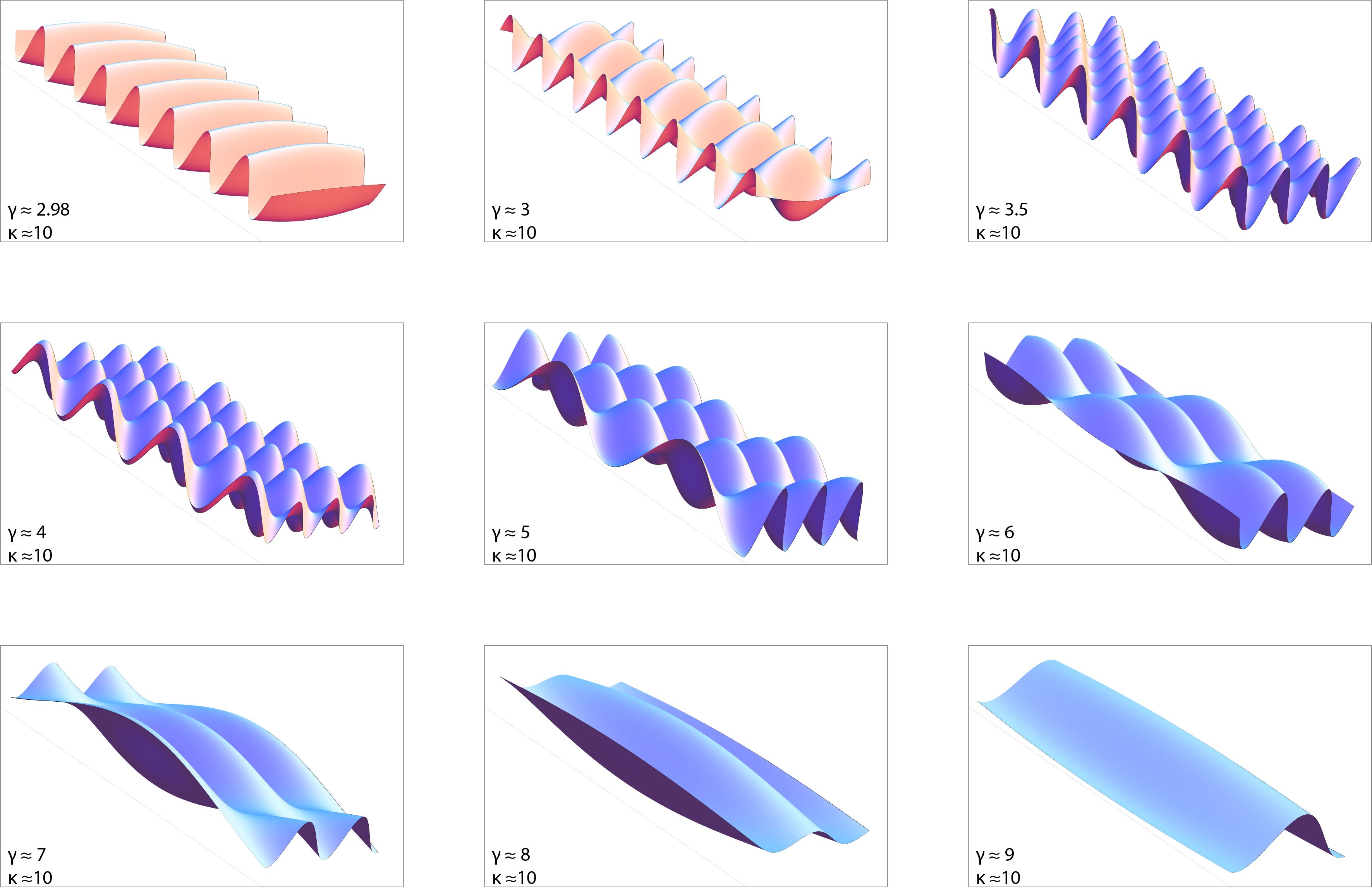}}
		\caption{Plotted above are the graphs (each having the same scale) of the free surface component within a selection of different roll wave solutions to system~\eqref{traveling form of the equations} (with $\mu=3/20$, $\sig=2$) covered by our main result, Theorem~\ref{1st main theorem}. This result applied in this situation is essentially telling us that if we fix (to leading order in the amplitude of the solution) the incline strength $\kappa=10$, then we can expect nontrivial traveling wave solutions to exist for relative wave speeds $\gam$ which are to leading order belonging to the range $\gam_{\m{min}}(10)\le
			\gam<10$ (see~\eqref{equivalent formulation for the set of interesting parameter values}). We choose 9 such values of $\gam$ and plot the corresponding example solutions. The periodicity lengths of these solutions are determined via the initial $(\gam,\kappa)$ though the function~\eqref{the period lengths}. Notice that the slowest waves are independent of the direction perpendicular to propagation while the faster waves have very long wavelength oscillations in the traveling direction.}
		\label{the roll waves plot}
	\end{figure}

	%*_*_*()(())()*_*_*()(())()*_*_*()(())()*_*_*()(())()*_*_*()(())()*_*_*()(())()*_*_*()(())()*_*_*()(())()
	\subsection{Survey of previous work}\label{subsection on survey of previous work}
	%*_*_*()(())()*_*_*()(())()*_*_*()(())()*_*_*()(())()*_*_*()(())()*_*_*()(())()*_*_*()(())()*_*_*()(())()
	
	We stress that there exists a plethora of versions of the shallow water equations, all of which are derived from the free boundary Euler or Navier-Stokes systems in one way or another; for various examples, we refer to~\cite{MR1324142,MR1821555,MR1975092,MR2118849,MR2281291,MR2397996,MR4105349}, and in particular the articles of Bresch~\cite{MR2562163} and Mascia~\cite{mascia_2010}.
	
	The roll wave phenomenon has been extensively studied in the mathematics and physics literature for over century and, as such, a complete survey is beyond the scope of this brief literature review. We shall content ourselves with focusing just on the results most closely related to our own. In all of the subsequent references, unless otherwise stated, the results concern a one dimensional shallow water model, either viscous or inviscid, neglecting capillary effects and with a Ch\'ezy-type drag term.
	
	Jeffreys~\cite{JEFF25} provided the first theoretical discussion of roll waves and analyzed the linear stability of inviscid flow over a flat plane. Dressler~\cite{MR0033717} established the existence of discontinuous roll wave solutions to the inviscid Saint-Venant equations. Brock~\cite{BROCK69} produced and empirically studied roll wave trains in laboratory flumes. Tougou~\cite{MR0570367} performs linear stability/instability analysis on the Dressler roll waves. Merkin and Needham~\cite{1984RSPSA.394..259N,MR0853213} add an energy dissipation term to the model used by Dressler and study the corresponding periodic roll wave phenomenon and stability questions. Hwang and Chang~\cite{MR0890282} discover a new family of roll wave solutions to the models of Dressler and Merkin-Needham. Kranenburg~\cite{MR1194982} numerically suggests instability for the viscous roll waves of Merkin and Needham under quasi-periodic perturbations and shows amplitude growth over time. Chang, Cheng, and Prokopiou~\cite{PCC91} derive a model which can include effects of surface tension and establish solitary roll waves. Kevorkian and Yu~\cite{MR1187443} examine, in a weakly unstable regime, weakly nonlinear small amplitude roll wave solutions to inviscid Saint-Venant equations and derive a model for the long time behavior. Mei and Ng~\cite{Ng_Mei_1994} develop a theory for one dimensional roll waves appearing in shallow layers of mud. Chang, Demekhinm and Kalaidin~\cite{MR1780437} study coherent structures and self-similarity in roll wave dynamics. Balmforth and Mandre~\cite{MR2259999} study the dynamics and stability of viscous roll waves and the interplay with varying the bottom topography. Nobel~\cite{MR2338351} establishes the existence of roll waves for a Saint-Venant equation with a periodically modulated bottom. A complete theory of linear and nonlinear stability of roll wave solutions to both the inviscid and viscous shallow water equations was developed by Johnson, Zumbrun, and Noble~\cite{MR2784868}, Barker, Johnson, Rodrigues, Zumbrun~\cite{MR2813829}, Barker, Johnson, Noble, Rodrigues, Zumbrun~\cite{MR3600832}, Johnson, Nobel, Rodrigues, and Zumbrun~\cite{MR3219516}, Rodrigues and Zumbrun~\cite{MR3449905}, and Johnson, Noble, Rodrigues, Yang, and Zumbrun~\cite{MR3933411}. For the state-of-the-art in the numerical and analytical theory of the two dimensional stability of inviscid roll waves, we refer the reader to the work of Yang and Zumbrun~\cite{yangZ2023}.
	
	Our results for the inclined shallow water system~\eqref{The time dependent inclined shallow water equation}, and its nondimensionalized traveling formulation~\eqref{traveling form of the equations}, are the first construction of properly two dimensional viscous roll wave solutions and show that a wide variety of new behavior is theoretically possible for thin fluids moving down an incline.
	
	We also make a connection to the general theory of traveling wave solutions to fluid equations. In the context of the free boundary Euler equations, this traveling wave study is also known as the water wave problem and a rich theory has been developing for over a century. For more information we refer the reader to the survey articles of Toland~\cite{Toland_1996}, Groves~\cite{Groves_2004}, Strauss~\cite{Strauss_2010}, and Haziot, Hur, Strauss, Toland, Wahl\'en, Walsh, and Wheeler~\cite{MR4406719}. In contrast, progress on the traveling wave problem for dissipative fluid models in dimensions at least 2 has only recently been made. Traveling waves generated by forcing have been studied as solutions to: the free boundary incompressible Navier-Stokes equations by Leoni and Tice~\cite{leoni2019traveling}, Stevenson and Tice~\cite{MR4337506,stevensontice2023wellposedness}, and Koganemaru and Tice~\cite{koganemaru2022traveling,koganemaru2023travelingwavesolutionsfree}; to the free boundary compressible Navier-Stokes equations by Stevenson and Tice~\cite{stevenson2023wellposedness}; to Darcy flow by Nguyen and Tice~\cite{nguyen_tice_2022}, Nguyen~\cite{nguyen2023largetravelingcapillarygravitywaves}, and Brownfield and Nguyen~\cite{brownfield2023slowlytravelinggravitywaves}; and to the damped shallow water equations by Stevenson and Tice~\cite{stevenson2023shallow}. Among this family of dissipative traveling wave problems is the common theme that nontrivial solutions have only been produced when a force in addition to gravity is supplied to the system. Therefore the result of this paper is the first example which does not follow this pattern - as the nontrivial roll wave solutions which are generated here exist in the presence of only a gravitational force.
	
	%*_*_*()(())()*_*_*()(())()*_*_*()(())()*_*_*()(())()*_*_*()(())()*_*_*()(())()*_*_*()(())()*_*_*()(())()
	\subsection{Main results and discussion}\label{subsection on main results and discussion}
	%*_*_*()(())()*_*_*()(())()*_*_*()(())()*_*_*()(())()*_*_*()(())()*_*_*()(())()*_*_*()(())()*_*_*()(())()
	
	The statements of this paper's main results require the introduction of some notation and helper functions. We shall work in classes of spatially periodic functions, viewing them as defined on flat 2-tori of various side lengths. If $\ell\in\R^+$ then the 1-torus of size $\ell$ is the set $\ell\T=\R/(\ell\Z)$. Now if $L=(L_1,L_2)\in\tp{\R^+}^2$, then the $2$-torus of side lengths $L_1$ and $L_2$ is the set $\T^2_L=\tp{L_1\T}\times\tp{L_2\T}$.
	
	The function spaces for our velocity and free surface functions are built from the following subspaces of Sobolev spaces which satisfy certain symmetry conditions. Firstly, we denote the vanishing average subspace via
	\begin{equation}
		{H^s_\bullet}(\T^2_L)=\tcb{f\in H^s(\T^2_L)\;:\;\mathscr{F}[f](0,0)=0}.
	\end{equation}
	Next we denote the subspaces of functions who are even or odd in the second variable:
	\begin{equation}
		H^s_{\pm}(\T^2_L)=\tcb{f\in H^s(\T^2_L)\;:\;f(x_1,x_2)=\pm f(x_1,-x_2),\;\forall\;x\in\T^2_L}.
	\end{equation}
	The spaces of functions for the velocity and free surface are denoted by
	\begin{equation}\label{the domain spaces}
		\bf{X}^{\m{velo}}_L=H^2_+(\T^2_L)\times H^2_-(\T^2_L),\quad\bf{X}^{\m{surf}}_L=(H^3_{\bullet}\cap H^3_{+})(\T^2_L),\quad \bf{X}_{L}=\bf{X}^{\m{velo}}_L\times\bf{X}^{\m{surf}}_L
	\end{equation}
	In other words we are restricting our attention to velocity fields which are second-argument even in the first component and second-argument odd in the second component along with free surface functions that are average zero and second argument even.
	
	The following (nonempty) set of non-dimensional relative wave speed and incline strength tuples is related to the classification of when the linearized problem associated with~\eqref{traveling form of the equations} has a kernel in periodic function spaces for some choice of period lengths:
	\begin{equation}\label{the set of interesting parameter values}
		\mathfrak{E}=\bcb{(\gam,\kappa)\in(1,\infty)^2\;:\;0<\kappa-\gam\le\f{4\mu}{\sig}\gam\tp{\gam^2-1}}.
	\end{equation}
	On the set $\mathfrak{E}$ we shall define numerous auxiliary functions: $\rho:\mathfrak{E}\to\R^+$, $\chi:\mathfrak{E}\to(0,1]$, and $\mathfrak{l}:\mathfrak{E}\to\tp{\R^+}^2$:
	\begin{equation}\label{the helper functions}
		\rho(\gam,\kappa)=\sqrt{\f{\kappa/\gam-1}{16\pi^2\mu}},\quad\chi(\gam,\kappa)=\f{1}{\gam}\sqrt{1+\f{\sig}{4\mu}\bp{\f{\kappa}{\gam}-1}},
	\end{equation}
	and
	\begin{equation}\label{the period lengths}
		\mathfrak{l}(\gam,\kappa)=\f{1}{\rho(\gam,\kappa)}\begin{cases}
			(1,1)&\text{if }\kappa-\gam=\f{4\mu}{\sig}\gam\tp{\gam^2-1},\\
			(1/\chi(\gam,\kappa),1/\sqrt{1-\chi(\gam,\kappa)^2})&\text{if }\kappa-\gam<\f{4\mu}{\sig}\gam\tp{\gam^2-1}.
		\end{cases}
	\end{equation}
	
	\begin{figure}[!h]
		\centering
		\scalebox{0.38}{\includegraphics{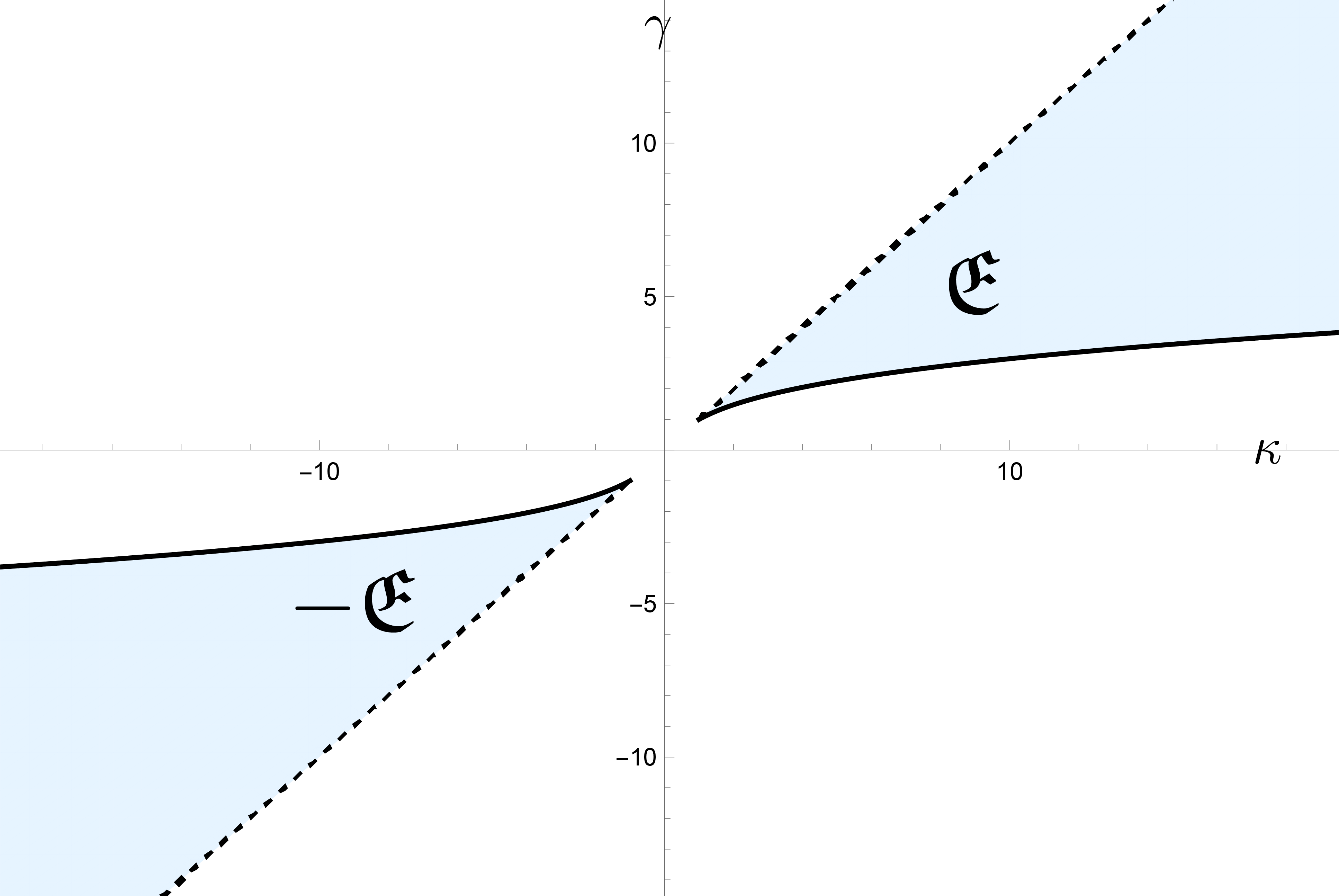}}
		\caption{Depicted here is a plot of $\mathfrak{E}\cup\tp{-\mathfrak{E}}$ for $\mu=3/20$, $\sig=2$, with wave speeds $\gam$ plotted on the vertical axis and tilt parameters $\kappa$ plotted on the horizontal axis. Theorems~\ref{1st main theorem} and~\ref{3rd main theorem} say roughly that small roll wave solutions exist for parameter combinations in the blue region and do not exist for combinations in the complement.}
		\label{the region of interesting combinations}
	\end{figure}
	
	Now we are ready to introduce our main results, whose proofs can be found in Section~\ref{subsection on conclusions}. We reemphasize that the inverse Reynolds number $\mu>0$ and the inverse Bond number $\sig>0$ are fixed throughout and all open sets and constants implicitly depend on these values.
	
	First we establish the existence of small roll wave solutions to the inclined shallow water equations with viscosity and surface tension for certain values of wave speed and incline. In other words, we show that nearby certain trivial solutions to~\eqref{traveling form of the equations} there exists a family of small nontrivial solutions to the equations provided that we vary the speed and tilt parameters. These roll wave families are semi-explicitly parametrized by the kernel of the linearized problem at the trivial solution and actually describe the totality of nearby small roll waves.
	
	\begin{customthm}{1}[Existence and multiplicity of small amplitude roll waves]\label{1st main theorem}
		Assume that $(\gam_\star,\kappa_\star)\in\mathfrak{E}\cup\tp{-\mathfrak{E}}$ and let $L=\mathfrak{l}(|\gam_\star|,|\kappa_\star|)\in\tp{\R^+}^2$ where $\mathfrak{l}$ is defined in~\eqref{the period lengths}. Then there exists a direct sum decomposition $\bf{X}_L=E\oplus Z$ with $\m{dim} E=2$, positive numbers $\ep,\bar{r}\in\R^+$, and an open set $(\gam_\star,\kappa_\star,0,0)\in\mathcal{O}\subset\R^2\times\bf{X}_L$ such that upon setting $K=E\cap\pd B_{\bf{X}_L}(0,1)$ the following hold.
		\begin{enumerate}
			\item There exists smooth functions $\tp{\tilde{u},\tilde{\eta}}:[-\ep,\ep]\times K\to \tp{B_{\bf{X}_L}(0,\bar{r})\cap Z}$ and $\tp{\tilde{\gam},\tilde{\kappa}}:[-\ep,\ep]\times K\to B_{\R^2}(0,\Bar{r})$ such that for all $(\al,\upupsilon,\upeta)\in[-\ep,\ep]\times K$ we have that upon defining
			\begin{equation}\label{nontrivial solutions}
				\begin{cases}
					(u,\eta)=\al\tp{\upupsilon,\upeta}+\al^2\tp{\tilde{u}\tp{\al,\upupsilon,\upeta},\tilde{\eta}\tp{\al,\upupsilon,\upeta}},\\
					(\gam,\kappa)=\tp{\gam_\star,\kappa_\star}+\al\tp{\tilde{\gam}\tp{\al,\upupsilon,\upeta},\tilde{\kappa}\tp{\al,\upupsilon,\upeta}},
				\end{cases}
			\end{equation}
			the velocity and free surface tuple $(u,\eta)$ is a solution to system~\eqref{traveling form of the equations} with speed and tilt parameters $(\gam,\kappa)$; moreover $\eta\neq0$ whenever $\al\neq0$.
			\item The only other local solutions are the trivial ones, more precisely:
			\begin{multline}
				\tcb{\tp{\bar{\gam},\bar{\kappa},\bar{u},\bar{\eta}}\in\mathcal{O}\;:\;\tp{\bar{\gam},\bar{\kappa},\bar{u},\bar{\eta}}\text{ is a solution to~\eqref{traveling form of the equations}}}\\=\tsb{\tcb{(\gam,\kappa,0,0)\;:\;(\gam,\kappa)\in\R^2}\cup\tcb{(\gam,\kappa,u,\eta)\;:\;\text{defined via }~\eqref{nontrivial solutions}}}\cap\mathcal{O}.
			\end{multline}
		\end{enumerate}
	\end{customthm}
	
	The next result is a simple corollary of Theorem~\ref{1st main theorem}; we read off to top order the shapes of the nontrivial free surface function and velocity components of the roll wave solutions. This theorem along with the previous one justifies the free surface functions depicted in Figure~\ref{the roll waves plot} and shows that the roll wave solutions we produce here are properly two dimensional.
	
	\begin{customthm}{2}[On the shape of roll waves]\label{2nd main thm}
		Under the hypotheses of Theorem~\ref{1st main theorem} consider the curve of solutions $\tcb{(\gam_\al,\kappa_\al,u_\al,\eta_\al)}_{\al\in[-\ep,\ep]}\in\bf{X}_L$ produced by the maps in~\eqref{nontrivial solutions} by varying $\al\in[-\ep,\ep]$ for some fixed $\tp{\upupsilon,\upeta}\in K$. There exists constants $A,B\in\R$ (not both zero) and $c\in\R^+$ such that
		\begin{enumerate}
			\item $\upeta$ has one of the formulae:
			\begin{enumerate}
				\item if $0<|\kappa_\star|-|\gamma_\star|=\f{4\mu}{\sig}|\gam_\star|\tp{|\gam_\star|^2-1}$, then
				\begin{equation}\label{border casez}
					\upeta(x_1,x_2)=A\cos(2\pi x_1/L_1)+B\sin(2\pi x_1/L_1),
				\end{equation}
				\item if $0<|\kappa_\star|-|\gamma_\star|<\f{4\mu}{\sig}|\gam_\star|\tp{|\gam_\star|^2-1}$
				\begin{equation}\label{interior casez}
					\upeta(x_1,x_2)=A\cos(2\pi x_1/L_1)\cos(2\pi x_2/L_2)+B\sin(2\pi x_1/L_1)\cos(2\pi x_2/L_2);
				\end{equation}
			\end{enumerate}
			\item $\upupsilon=(\upupsilon_1,\upupsilon_2)$ is determined via $\upeta$ through the formulae
			\begin{equation}\label{component 1}
				\upupsilon_1=-\kappa_\star\tp{1-\gam\pd_1-\mu\Delta}^{-1}\mathcal{R}_2^2\upeta-\gam_\star\mathcal{R}_1^2\upeta
			\end{equation}
			and
			\begin{equation}\label{component 2}
				\upupsilon_2=\kappa_\star\tp{1-\gam\pd_1-\mu\Delta}^{-1}\mathcal{R}_{1}\mathcal{R}_2\upeta-\gam_\star\mathcal{R}_1\mathcal{R}_2\upeta,
			\end{equation}
			where $\mathcal{R}_1$ and $\mathcal{R}_2$ are the Riesz transforms in the $e_1$ and $e_2$ directions, respectively;
			\item and we have for all for all $\al\in[-\ep,\ep]$:
			\begin{equation}\label{deviation from first order approximation}
				\tnorm{\eta_\al-\al\upeta}_{L^\infty}+\tnorm{u_\al-\al\upupsilon}_{L^\infty}\le c\al^2.
			\end{equation}
		\end{enumerate}
	\end{customthm}
	
	In our final result we consider the speed and tilt parameters lying in the complement of the region considered in Theorem~\ref{1st main theorem} and prove that in this case small nontrivial roll waves do not exist, even within the larger class of functions without second variable parity assumptions. Figure~\ref{the region of interesting combinations} allows one to visualize the parameters in tilt-speed space from which it is possible for small nontrivial roll wave solution families to emanate.
	
	\begin{customthm}{3}[Nonexistence of roll waves]\label{3rd main theorem}
		Suppose that $(\gam_\star,\kappa_\star)\in\R^2\setminus\tp{\mathfrak{E}\cup\tp{-\mathfrak{E}}}$ and let $L\in\tp{\R^+}^2$ be arbitrary. There exists an $\ep\in\R^+$, depending on $\mu$ and $\sig$ in addition to the aforementioned parameters, such that if
		\begin{equation}
			(\gam,\kappa,u,\eta)\in\R^2\times H^2(\T^2_L;\R^2)\times H^3_\bullet(\T^2_L)\text{ satisfies }|\gam-\gam_\star|+|\kappa-\kappa_\star|+\tnorm{u}_{H^2}+\tnorm{\eta}_{H^3}\le\ep
		\end{equation}
		and is a solution to system~\eqref{traveling form of the equations}, then $u=0$ and $\eta=0$.
	\end{customthm}
	
	Now that the main results of this paper have been presented, we shall discuss a few remarks. A high level summary of our results is as follows. For Theorems~\ref{1st main theorem} and~\ref{2nd main thm}, we work in a class of functions obeying the following symmetry conditions in the second spatial variable: the free surface and the first component of the velocity are even while the second component of the velocity is odd; on the other hand, for Theorem~\ref{3rd main theorem} no parity assumptions are necessary. We classify for which initializations $(\gam_\star,\kappa_\star,0,0)$ there exists or does not exist period lengths $(L_1,L_2)$ for the spatial variables and arbitrarily close periodic and nontrivial tuples $(\gam_\star,\kappa_\star,0,0)+\al(\tilde{\gam}_\al,\tilde{\kappa}_\al,\Bar{u}_\al,\Bar{\eta}_\al)$ solving system~\eqref{traveling form of the equations} for a sequence $\al\to0$.
	
	It is proved that if $(\gam_\star,\kappa_\star)$ belong to the region $\mathfrak{E}\cup\tp{-\mathfrak{E}}$ (see~\eqref{the set of interesting parameter values} and Figure~\ref{the region of interesting combinations}), then there exists such curves of nontrivial solutions with terminus $(\gam_\star,\kappa_\star,0,0)$, if one selects the period lengths $(L_1,L_2)$ to depend on the initial wave speed and tilt through the function~\eqref{the period lengths}. Moreover, the totality of small amplitude solutions to~\eqref{traveling form of the equations} nearby the trivial one $(\gam_\star,\kappa_\star,0,0)$, within the class obeying the aforementioned second variable symmetry assumptions and prescribed periodicity, is described. This is the content of Theorem~\ref{1st main theorem}.
	
	The set of solutions $(\gam,\kappa,u,\eta)$ generated with terminus $(\gam_\star,\kappa_\star)\in\mathfrak{E}$ are in a simple reflective correspondence with the set of solutions $(\hat{\gamma},\hat{\kappa},\hat{u},\hat{\eta})$ generated with the negative terminus $(-\gam_\star,-\kappa_\star)\in-\mathfrak{E}$, specifically the mapping $(-\hat{\gam},-\hat{\kappa},R\hat{u}\circ R,\hat{\eta}\circ R)=(\gam,\kappa,u,\eta)$, with $R(y_1,y_2)=(-y_1,y_2)$ being the reflection operator, gives a bijective correspondence between these two sets of solutions (see equations~\eqref{reflection equation 1} and~\eqref{reflection equation 2} for more details.). Note that appeals to one's physical intuition as one expects the collection of leftward traveling roll waves to be in symmetric correspondence with the collection of rightward traveling ones.
	
	One could equivalently formulate the set $\mathfrak{E}$ of definition~\eqref{the set of interesting parameter values} in terms of the minimal traveling speed for the linearized problem at a fixed tilt strength. Specifically, we can let $\gam_{\m{min}}(\kappa)\in(1,\kappa)$ be the unique number satisfying the equation $\tp{\gam_{\m{min}}(\kappa)}^2=1+\f{\sig}{4\mu}\sp{\f{\kappa}{\gam_{\m{min}}(\kappa)}-1}$ and deduce that
	\begin{equation}\label{equivalent formulation for the set of interesting parameter values}
		\mathfrak{E}=\tcb{(\gam,\kappa)\in(1,\infty)^2\;:\;\gam_{\m{min}}(\kappa)\le\gam<\kappa}.
	\end{equation}
	We now see that it is suggested by Theorem~\ref{2nd main thm} that the `slowest' wave families - i.e. the one's generated with a terminus $(\gam_\star,\kappa_\star)\in\mathfrak{E}\cup\tp{-\mathfrak{E}}$ satisfying $|\gam_\star|=\gam_{\m{min}}(|\kappa_\star|)$, which are covered by the first case (i.e.~\eqref{border casez}) in the theorem - are one dimensional in the sense that they do not depend on the second variable. This is indeed the case since the kernel of the linearized problem in this case consists of functions of just the first variable; so a minor modification of the proof of Theorem 1 will show that the families of solutions generated in this case are one dimensional functions. We do not pursue the full justification of this fact, however, as our primary interest is in the generation of properly two dimensional roll waves. Theorem~\ref{2nd main thm} also tells us, specifically in the second conclusion (i.e.~\eqref{interior casez}), that all families of roll waves generated by a terminus $(\gam_\star,\kappa_\star)\in\mathfrak{E}\cup\tp{-\mathfrak{E}}$ with an initial wave speed exceeding the minimal value, i.e. $|\gam_{\star}|>\gam_{\m{min}}(|\kappa_\star|)$, are necessarily properly two dimensional.
 
 On the other hand, one observes that each component of $\mathfrak{l}(|\gam_\star|,|\kappa_\star|)$ (the period lengths function from~\eqref{the period lengths}) tends to $\infty$ as $|\gam_\star|\to|\kappa_\star|$. So the `faster' roll waves produced by Theorem~\ref{1st main theorem} are extremely slowly varying in space.
	
	We also remark that while the roll wave solutions produced by Theorem~\ref{1st main theorem} live initially in the regularity classes $(u,\eta)\in H^2(\T^2_L;\R^2)\times H^3(\T^2_L)$, a straightforward a posteriori regularity promotion argument (omitted here for brevity) will show that these solutions are, in fact, smooth.
	
	Theorem~\ref{3rd main theorem} considers what happens in the opposite case of when $(\gam,\kappa)$ do not belong to the region $\mathfrak{E}\cup\tp{-\mathfrak{E}}$. Here it is simply established that the trivial solution is locally unique (no matter the choice of the spatial period lengths) and so there does not exist arbitrarily small amplitude roll waves in our functional framework for these combinations of wave speed and tilt. In particular, this rules out two type of solutions: there does not exist arbitrarily small roll waves for small tilt ($|\kappa|<1$) and there does not exist arbitrarily small roll waves which `travel uphill'. Indeed, a positive $\kappa$ indicates tilting down to the right while negative $\kappa$ corresponds to tilting down to the left. Recall that $\Bar{\gam}=\gam+\kappa$, defined in~\eqref{traveling ansatz}, is the dimensionless speed of the traveling frame, with $\Bar{\gam}>0$ indicating a rightward traveling wave and $\Bar{\gam}<0$ indicating a leftward traveling wave. So if $\Bar{\gam}$ and $\kappa$ have opposite signs (which happens in either traveling uphill case), then $(\gam,\kappa)$ necessarily belongs to the complement of $\mathfrak{E}\cup\tp{-\mathfrak{E}}$ and hence there are not arbitrarily small roll waves in this case.
	
	This construction of the first properly two dimensional roll wave solutions to~\eqref{traveling form of the equations} opens the door to several lines of further inquiry which are delayed for future work. For instance, one can ask: For which values of tilt strength $\kappa$ is the equilibrium solution stable or unstable? Can the curves of small solutions be extended to large amplitudes? Which other fluid models admit similar families of nontrivial solutions?
	
	To close this subsection, we shall briefly discuss the proof strategy for our main theorems and outline the paper. Theorem~\ref{3rd main theorem} is established via a simple inverse function theorem argument: the invertibility of the linearization is proved via explicitly inverting the symbol of the linearized PDE, which is possible for speed and tilt parameters in $\R^2\setminus\tp{\mathfrak{E}\cup\tp{-\mathfrak{E}}}$. Theorem~\ref{2nd main thm} is actually a straightforward consequence of Theorem~\ref{1st main theorem} and its proof.
	
	Theorem~\ref{1st main theorem} is proved via a local bifurcation theory argument. Upon study of the linearized operator associated with~\eqref{traveling form of the equations}, which we call $P(\gam,\kappa,\cdot,\cdot):\bf{X}_L\to\bf{Y}_L$ in~\eqref{the family of linearized operators P}, one finds that for $(\gam,\kappa)\in\mathfrak{E}$ and $L=\mathfrak{l}(\gam,\kappa)$ (see~\eqref{the period lengths}) the kernel is a two dimensional subspace in $\bf{X}_L$. On the other hand $P(\gam,\kappa,\cdot,\cdot)$ is Fredholm of index zero and hence has a closed range in $\bf{Y}_L$ of codimension 2. It is therefore not clear how our functional framework would support a classical `simple eigenvalue'-type bifurcation argument as in Crandall and Rabinowitz~\cite{MR0288640}. Instead we look to a multiparameter bifurcation style argument to make up for the codimension 2 range of the linearization.
	
	The precise abstract bifurcation tool utilized in this work, which will not be a surprise to experts in the area, is recorded with proof, for clarity and convenience, in Section~\ref{subsection on abstract bifurcation}. A straightforward synthesis of the ideas in Theorem I.19.6 in Kielh\"ofer~\cite{MR2859263} and Lemma 1.12 in Crandall and Rabinowitz~\cite{MR0288640} yields a multiparameter bifurcation theorem which is capable of classifying all small solutions.
	
	In Section~\ref{section on linear analysis} we aim to verify the `linear hypotheses' of the bifurcation theorem. Section~\ref{subsection on analysis of the kernel} computes the kernel of $P(\gam,\kappa,\cdot,\cdot)$ in terms of the kernel of a related scalar partial differential operator. This latter operator can be readily analyzed with Fourier analysis and it is found that its kernel is exactly the subspace of functions frequency supported on a special set of modes. In Section~\ref{subsection on analysis of the range} we then turn our attention to the range of the operator $P(\gam,\kappa,\cdot,\cdot)$. We show that the variation of $P$ in the speed and tilt parameters acting on nontrivial members of the kernel makes up the missing two dimensions of the range; in other words a transversality condition is satisfied. This leads to the satisfaction of the remaining linear hypotheses of the bifurcation tool, Theorem~\ref{theorem on multiparameter bifurcation}. Section~\ref{subsection on synthesis of the linear analysis} is a synthesis of the linear analysis.
	
	Section~\ref{section on nonlinear analysis} first develops the abstract bifurcation tool, Theorem~\ref{theorem on multiparameter bifurcation},  with the content of Section~\ref{subsection on abstract bifurcation}. Next, in Section~\ref{subsection on concrete nonlinear analysis}, we perform some simple smoothness verification for the shallow water system's nonlinearities, thereby checking the nonlinear bifurcation hypotheses. Finally, in Section~\ref{subsection on conclusions}, we combine all of the previous material into proofs of the main results, Theorems~\ref{1st main theorem}, \ref{2nd main thm}, and~\ref{3rd main theorem}.

	%*_*_*()(())()*_*_*()(())()*_*_*()(())()*_*_*()(())()*_*_*()(())()*_*_*()(())()*_*_*()(())()*_*_*()(())()
	\subsection{Conventions of notation and function spaces}\label{subsection on CoN & FS}
	%*_*_*()(())()*_*_*()(())()*_*_*()(())()*_*_*()(())()*_*_*()(())()*_*_*()(())()*_*_*()(())()*_*_*()(())()
	
	$\N$ is the set $\tcb{0,1,2,\dots}$ while $\N^+=\N\setminus\tcb{0}$. We denote $\R^+=(0,\infty)$. The integers are denoted by $\Z=\tcb{\dots,-2,-1,0,1,2,\dots}$. The notation $A\lesssim B$ means that there exists a constant $C\in\R^+$, depending on the parameters which are clear from the context, such that $A\le CB$. We also express that two quantities $A$ and $B$ are equivalent, written $A\asymp B$, if $A\lesssim B$ and $B\lesssim A$. We shall use the bracket notation: $\tbr{\cdot}:\C^k\to\R^+$:
	\begin{equation}
		\tbr{x}=\sqrt{1+|x_1|^2+\cdots+|x_k|^2},\quad x=(x_1,\dots,x_k)\in\C^k.
	\end{equation}
	
	The Fourier transform of a function $f:\T^2_L\to\C^k$ is denoted by $\mathscr{F}[f]:\tp{\Z/L_1}\times\tp{\Z/L_2}\to\C^k$ and has the formula
	\begin{equation}
		\mathscr{F}[f](\xi_1,\xi_2)=\f{1}{\sqrt{L_1L_2}}\int_0^{L_1}\int_{0}^{L_2}f(x_1,x_2)e^{-2\pi\ii x_1\xi_1}e^{-2\pi\ii x_2\xi_2}\;\m{d}x_1\;\m{d}x_2
	\end{equation}
	while the corresponding inverse Fourier transform of a sequence $\varphi:(\Z/L_1)\times\tp{\Z/L_2}\to\C^k$ is the function $\mathscr{F}^{-1}[\varphi]:\T^2_L\to\C^k$ and is given via
	\begin{equation}
		\mathscr{F}^{-1}[\varphi](x_1,x_2)=\f{1}{\sqrt{L_1L_2}}\sum_{(\xi_1,\xi_2)\in\Z/L_1\times\Z/L_2}\varphi(\xi_1,\xi_2)e^{2\pi\ii x_1\xi_1}e^{2\pi\ii x_2\xi_2}.
	\end{equation}
	With these definitions the Fourier reconstruction formula
	\begin{equation}
		f(x_1,x_2)=\f{1}{\sqrt{L_1L_2}}\sum_{(\xi_1,\xi_2)\in\Z/L_1\times\Z/L_2}\mathscr{F}[f](\xi_1,\xi_2)e^{2\pi\ii x_1\xi_1}e^{2\pi\ii x_2\xi_2}
	\end{equation}
	holds.
	
	The vector of Riesz transforms, denoted by $\mathcal{R}$, is the Fourier multiplication operator with the symbol vanishing at the origin and satisfying $\xi\mapsto\ii|\xi|^{-1}\xi$ for $\xi\neq 0$. In other words $\mathcal{R}=|\grad|^{-1}\grad$. The Leray projection operator onto divergence free vector fields, denoted $\mathbb{P}$, is the Fourier multiplication operator with the symbol equal to the identity at the origin and obeying $\xi\mapsto I-\f{\xi\otimes\xi}{|\xi|^2}$ for $\xi\neq0$. Note that $\mathbb{P}=I+\mathcal{R}\otimes\mathcal{R}$.
	
	For $s\in\N$ we let $H^s(\T^2_L;\R^k)$ denote the standard $L^2$-based Sobolev space of $\R^k$-valued functions with $s$ (weak) derivatives in $L^2$. When $k=1$ we shall denote $H^s(\T^2_L;\R)=H^s(\T^2_L)$. As a norm on this space we take
	\begin{equation}
		\tnorm{f}_{H^s}^2=\sum_{\xi\in\Z/L_1\times\Z/L_2}\tbr{\xi}^{2s}|\mathscr{F}[f](\xi)|^2.
	\end{equation}

	The notation for our domain spaces is introduced in the beginning of Section~\ref{subsection on main results and discussion}. The notation for the codomain spaces is similar. We pose the momentum and continuity equations of~\eqref{traveling form of the equations} in the spaces
	\begin{equation}\label{the codomain spaces}
		\bf{Y}^{\m{mome}}_L=H^0_+(\T^2_L)\times H^0_{-}(\T^2_L),\quad\bf{Y}^{\m{cont}}_L=\tp{H^1_{\bullet}\cap H^1_+}(\T^2_L),\quad \bf{Y}_L=\bf{Y}^{\m{mome}}_L\times\bf{Y}^{\m{cont}}_L.
	\end{equation}
	
	Finally we shall introduce notation for larger container spaces for the domain $\bf{X}_L$ and the codomain $\bf{Y}_L$ which do not enforce the second argument parity assumptions. Let
	\begin{equation}\label{larger than life spaces}
		\Bar{\bf{X}}_L=H^2(\T^2_L;\R^2)\times H^3_\bullet(\T^2_L),\quad\Bar{\bf{Y}}_L= H^0(\T^2_L;\R^2)\times H^1_\bullet(\T^2_L).
	\end{equation}
	Observe that $\bf{X}_L\subset\Bar{\bf{X}}_L$ and $\bf{Y}_L\subset\Bar{\bf{Y}}_L$ are closed subspaces.
	
	%*_*_*()(())()*_*_*()(())()*_*_*()(())()*_*_*()(())()*_*_*()(())()*_*_*()(())()*_*_*()(())()*_*_*()(())()
	\section{Linear Analysis}\label{section on linear analysis}
	%*_*_*()(())()*_*_*()(())()*_*_*()(())()*_*_*()(())()*_*_*()(())()*_*_*()(())()*_*_*()(())()*_*_*()(())()
	
	The goal of this section is to study the linearized problem corresponding to system~\eqref{traveling form of the equations}. For $L=(L_1,L_2)\in\tp{\R^+}^2$ we consider the following family of linear maps
	\begin{equation}\label{the family of linearized operators P}
		P:\R^2\times\Bar{\bf{X}}_L\to\Bar{\bf{Y}}_L,\quad P(\gam,\kappa,u,\eta)=\bpm-\gam\pd_1 u+u-\mu\grad\cdot\mathbb{S}u+\grad(1-\sig\Delta)\eta-\kappa\eta e_1\\-\gam\pd_1\eta+\grad\cdot u\epm,
	\end{equation}
	where we recall that the function spaces $\Bar{\bf{X}}_L$ and $\Bar{\bf{Y}}_L$ are defined in~\eqref{larger than life spaces}. We note that the operator $P$ is parity preserving in the sense that for all $(\gam,\kappa,u,\eta)\in\R^2\times\bf{X}_L$ it holds that $P(\gam,\kappa,u,\eta)\in\bf{Y}_L$, where these spaces are defined in equations~\eqref{the domain spaces} and~\eqref{the codomain spaces}, respectively. So we may also view $P$ as a function $P:\R^2\times\bf{X}_L\to\bf{Y}_L$.
	
	%*_*_*()(())()*_*_*()(())()*_*_*()(())()*_*_*()(())()*_*_*()(())()*_*_*()(())()*_*_*()(())()*_*_*()(())()
	\subsection{Analysis of the kernel}\label{subsection on analysis of the kernel}
	%*_*_*()(())()*_*_*()(())()*_*_*()(())()*_*_*()(())()*_*_*()(())()*_*_*()(())()*_*_*()(())()*_*_*()(())()
	
	Here we explore the kernels of the operators $P$ from~\eqref{the family of linearized operators P} for certain values of $(\gam,\kappa)$ and $L$. Our first task is to reduce the computation of the kernel of $P$ to that of a related, but scalar, integro-differential operator $Q$, which we define next. By the map $Q:\R^2\times H^3_\bullet(\T^2_L)\to H^3_\bullet(\T^2_L)$ we shall mean
	\begin{equation}\label{the family of auxilary linearized operators Q}
		Q(\gam,\kappa,\eta)=\Delta^{-2}\tsb{\tp{\tp{\gam-\kappa}-4\mu\gam\Delta}\pd_1\eta+\tp{\Delta-\gam^2\pd_1^2-\sig\Delta^2}\eta},\quad(\gam,\kappa,\eta)\in\R^2\times H^3_\bullet(\T^2_L),
	\end{equation}
	One can easily check that the operators $Q$ are also parity preserving, as they map $Q:\R^2\times\bf{X}^{\m{surf}}_L\to\bf{X}^{\m{surf}}_L$, with $\bf{X}^{\m{surf}}_L$ defined in~\eqref{the domain spaces}.
	
	In this subsection and the next, unless otherwise stated, we are viewing $P(\gam,\kappa,\cdot,\cdot):\bf{X}_L\to\bf{Y}_L$ and $Q(\gam,\kappa,\cdot):\bf{X}_L^{\m{surf}}\to\bf{X}_L^{\m{surf}}$; e.g. when talking about linear algebraic constructions like the kernel and range the context is for these restricted operators which enforce the parity assumptions.
	\begin{prop}[Correspondence of kernels]\label{proposition on the reduction of kernels}
		For $L\in\tp{\R^+}^2$ and $(\gam,\kappa)\in\R^2$ the restriction of the linear map $R_{\gam,\kappa}:\bf{X}_L^{\m{surf}}\to\bf{X}_L$ given via
		\begin{equation}\label{the kernel isomorphism operator}
			R_{\gam,\kappa}\eta=\bpm\kappa(1-\gam\pd_1-\mu\Delta)^{-1}(I+\mathcal{R}\otimes\mathcal{R})\tp{\eta e_1}-\gam\mathcal{R}\otimes\mathcal{R}\tp{\eta e_1}\\\eta\epm
		\end{equation}
		to $\m{ker}Q(\gam,\kappa,\cdot)\subset\bf{X}_L^{\m{surf}}$ takes values in $\m{ker}P(\gam,\kappa,\cdot,\cdot)\subset\bf{X}_L$ and gives an isomorphism $\m{ker}Q(\gam,\kappa,\cdot)\to\m{ker}P(\gam,\kappa,\cdot,\cdot)$.
	\end{prop}
	\begin{proof}
		The mapping $R_{\gam,\kappa}$ of~\eqref{the kernel isomorphism operator} is evidently continuous and an injection, which means we need only to check that its restriction to $\m{ker}Q(\gam,\kappa,\cdot)$ has image $\m{ker}P(\gam,\kappa,\cdot,\cdot)$.
		
		Suppose first that $(u,\eta)\in\m{ker}P(\gam,\kappa,\cdot,\cdot)\subset\bf{X}_L$. The equations in the first and second components of $P$ in~\eqref{the family of linearized operators P} are then zero. We decompose the vector field $u=v+w$, $v=(I-\mathbb{P})u$ and $w=\mathbb{P}u$ into its potential and solenoidal parts. The second equation in $P$ reveals to us that $\gam\pd_1\eta=\grad\cdot u$ and hence $v=\gam(I-\mathbb{P})\tp{\eta e_1}=-\gam\mathcal{R}\otimes\mathcal{R}\tp{\eta e_1}$.
		
		Upon returning to the first component of $P$ and applying $\mathbb{P}$, we derive that $(1-\gam\pd_1-\mu\Delta)w=\kappa\mathbb{P}\tp{\eta e_1}$ and hence $w=\kappa(1-\gam\pd_1\Delta)^{-1}\tp{I+\mathcal{R}\otimes\mathcal{R}}\tp{\eta e_1}$. After recalling that $u=v+w$ we see that at this point we have established that $R_{\gam,\kappa}\eta=(u,\eta)$.
		
		The goal now is to establish that $\eta\in\m{ker}Q(\gam,\kappa,\cdot)\subset\bf{X}_L^{\m{surf}}$. We apply the divergence to the first component equation of $P$ and substitute that $\grad\cdot u=\gam\pd_1\eta$ to derive:
		\begin{equation}\label{what it takes to see}
			\tp{\tp{\gam-\kappa}-4\gam\mu\Delta}\pd_1\eta+\tp{\Delta-\gam^2\pd_1^2-\sig\Delta^2}\eta=0.
		\end{equation}
		Since $\eta$ has vanishing average, we can apply $\Delta^{-2}$ to~\eqref{what it takes to see} and obtain that $\eta\in\m{ker}Q(\gam,\kappa,\cdot)$.
		
		To complete the proof, we shall argue now that if $\eta\in\m{ker}Q(\gam,\kappa,\cdot)$, then $(u,\eta)=R_{\gam,\kappa}\eta\in\m{ker}P(\gam,\kappa,\cdot)$. From the definition of $R_{\gam,\kappa}$, it is clear that $\grad\cdot u=\gam\pd_1\eta$ so that the second component of $P(\gam,\kappa,u,\eta)$ vanishes. In order to show that the first component of $P(\gam,\kappa,\cdot,\cdot)$ also vanishes, we can decompose it into potential and solenoidal parts with the Leray projector. As we calculated in the previous step, the solenoidal part is the expression: $(1-\gam\pd_1-\mu\Delta)\mathbb{P}u-\kappa\mathbb{P}\tp{\eta e_1}$. The definition $(u,\eta)=R_{\gam,\kappa}\eta$ ensures that this vanishes. The potential part is the expression: $\Delta^{-1}\grad\tsb{\tp{\tp{\gam-\kappa}-4\mu\gam\Delta}\pd_1\eta+\tp{\Delta-\gam^2\pd_1^2-\sig\Delta^2}\eta}=\Delta\grad Q(\gam,\kappa,\eta)$ which also vanishes.
	\end{proof}
	
	Our next goal is to compute the kernel of the operator $Q$. Thanks to the previous result, once this is complete, we obtain the kernel of $P$ via application of the map $R_{\gam,\kappa}$. In the next result we recall the notation introduced in equations~\eqref{the set of interesting parameter values}, \eqref{the helper functions}, and~\eqref{the period lengths}. We also consider the set valued function
	\begin{equation}\label{the frequencies of the kernel}
		V(\gam,\kappa)=\tcb{\rho(\gam,\kappa)\tp{\sig_1\chi(\gam,\kappa),\sig_2\sqrt{1-\chi(\gam,\kappa)^2}}\;:\;\sig_1,\sig_2\in\tcb{-1,1}}.
	\end{equation}
	Notice that $V(\gam,\kappa)$ is a subset of the circle of radius $\rho(\gam,\kappa)$ that has 2 points in the case that $\kappa-\gam=\f{4\mu}{\sig}\gam\tp{\gam^2-1}$ and has 4 points in the opposite case of $\kappa-\gam<\f{4\mu}{\sig}\gam\tp{\gam^2-1}$. Note that the function $\mathfrak{l}$ has the property that $V(\gam,\kappa)\subset\tp{\Z/L_1\times \Z/L_2}\setminus\tcb{0}$ with $(L_1,L_2)=\mathfrak{l}(\gam,\kappa)$.

	\begin{prop}[Computation of the reduced kernel]\label{proposition on computation of the reduced kernel}
		Suppose that $(\gam,\kappa)\in\mathfrak{E}$ and that $L=\mathfrak{l}(\gam,\kappa)\in\tp{\R^+}^2$. The following hold.
		\begin{enumerate}
			\item A function $\eta\in\bf{X}^{\m{surf}}_L$ belongs to $\m{ker}Q(\gam,\kappa,\cdot)$ if and only if $\mathscr{F}[\eta]$ is supported on the set $V(\gam,\kappa)\subset\tp{\Z/L_1\times\Z/L_2}\setminus\tcb{0}$.
			\item The subspace $\m{ker}Q(\gam,\kappa,\cdot)\subset\bf{X}^{\m{surf}}_L$ is two dimensional and is spanned by the smooth functions $\upvarphi_+$ and $\upvarphi_-$ whose formulae are given by:
			\begin{enumerate}
				\item In the case that $0<\kappa-\gam=\f{4\mu}{\sig}\gam\tp{\gam^2-1}$
				\begin{equation}\label{border case}
					\upvarphi_\iota(x)=\sqrt{\f{2}{L_1L_2}}\begin{cases}
						\cos(2\pi x_1/L_1)&\text{if }\iota=+,\\
						\sin(2\pi x_1/L_1)&\text{if }\iota=-.
					\end{cases}
				\end{equation}
				\item In the case that $0<\kappa-\gam<\f{4\mu}{\sig}\gam\tp{\gam^2-1}$
				\begin{equation}\label{interior case}
					\upvarphi_\iota(x)=\sqrt{\f{4}{L_1L_2}}
					\begin{cases}
						\cos(2\pi x_1/L_1)\cos(2\pi x_2/L_2)&\text{if }\iota=+,\\
						\sin(2\pi x_1/L_1)\cos(2\pi x_2/L_2)&\text{if }\iota=-.
					\end{cases}
				\end{equation}
			\end{enumerate}
		\end{enumerate}
	\end{prop}
	\begin{proof}
		We shall first prove the first item. Define the auxiliary set
		\begin{equation}\label{the set of interesting frequencies}
			\tilde{V}(\gam,\kappa)=\bcb{\xi\in\tp{\Z/L_1\times\Z/L_2}\setminus\tcb{0}\;:\;\bpm1-\f{\kappa}{\gam}+16\pi^2\mu|\xi|^2\\|\xi|^2-\gam^2\xi_1^2+4\pi^2\sig|\xi|^4\epm=0}.
		\end{equation}
		For $\eta\in\bf{X}_L^{\m{surf}}$ it holds (by taking Fourier transforms) that $Q(\gam,\kappa,\eta)=0$ if and only if for all $\xi\in\Z/L_1\times\Z/L_2$ we have
		\begin{equation}\label{the basic kernel of Q identity}
			\tsb{2\pi\ii\xi_1\tp{\tp{\gam-\kappa}+16\pi^2\mu\gamma|\xi|^2}-4\pi^2\tp{|\xi|^2-\gam^2\xi_1^2+4\pi^2\sig|\xi|^4}}\mathscr{F}[\eta](\xi)=0.
		\end{equation}
		It is then immediate that if $\mathscr{F}[\eta]$ is supported in the set $\tilde{V}(\gam,\kappa)$, then $\eta\in\m{ker}Q(\gam,\kappa,\cdot)\subset\bf{X}_{L}^{\m{surf}}$. In fact the opposite is true as well which we shall now establish. So let us assume that $\eta\in\m{ker}Q(\gam,\kappa,\cdot)$. As~\eqref{the basic kernel of Q identity} must hold, we see that $\mathscr{F}[\eta]$ is necessarily supported on frequencies $\xi\in\tp{\Z/L_1\times\Z/L_2}\setminus\tcb{0}$ such that $\tsb{2\pi\ii\xi_1\tp{\tp{\gam-\kappa}+16\pi^2\mu\gamma|\xi|^2}-4\pi^2\tp{|\xi|^2-\gam^2\xi_1^2+4\pi^2\sig|\xi|^4}}=0$, which (by taking real and imaginary parts) is equivalent to saying
		\begin{equation}
			\xi\in\m{supp}\mathscr{F}[\eta]\imp\bpm\xi_1\tp{1-\f{\kappa}{\gam}+16\pi^2\mu|\xi|^2}\\|\xi|^2-\gam^2\xi_1^2+4\pi^2\sig|\xi|^4\epm=0.
		\end{equation}
		We therefore obtain that $\m{supp}\mathscr{F}[\eta]\subseteq\tilde{V}(\gam,\kappa)$ as soon as we establish that $\xi\in\m{supp}\mathscr{F}[\eta]$ implies that $\xi_1\neq 0$. But this fact can be read off from identity~\eqref{the basic kernel of Q identity}: if $\xi_1=0$, then $\tp{|\xi|^2+4\pi^2\sig|\xi|^4}\mathscr{F}[\eta](\xi)=0$ and hence (since $\xi\neq0$) $\mathscr{F}[\eta](\xi)=0$.
		
		Thus-far we have proved that $\eta\in\m{ker}Q(\gam,\kappa,\cdot)$ if and only if $\m{supp}\mathscr{F}[\eta]\subseteq\tilde{V}(\gam,\kappa)$. The first item will follow once we prove that $\tilde{V}(\gam,\kappa)=V(\gam,\kappa)$ (where the latter set is defined in equation~\eqref{the frequencies of the kernel}). Directly from the definition, for $\xi\in\tp{\Z/L_1\times\Z/L_2}\setminus\tcb{0}$ we deduce that $\xi\in\tilde{V}(\gam,\kappa)$ if and only if
		\begin{equation}
			\tabs{\xi}^2=\f{\kappa/\gam-1}{16\pi^2\mu}=\rho(\gam,\kappa)^2,\quad\xi_1^2=\f{1}{\gam^2}|\xi|^2\tp{1+4\pi^2\sig|\xi|^2}=\rho(\gam,\kappa)^2\chi(\gam,\kappa)^2,
		\end{equation}
		where we recall that the functions $\rho$ and $\kappa$ are defined in equation~\eqref{the helper functions}. We note that it is here where we are using the assumption that $(\gam,\kappa)\in\mathfrak{E}$ to ensure that $\gam$ and $\kappa$ satisfy inequalities consistent with $|\xi|^2>0$ and $\xi_1^2\le|\xi|^2$ and hence the numbers $\rho(\gam,\kappa)$ and $\chi(\gam,\kappa)$ are well-defined.
		
		By calculating the square of the second component of a vector in terms of the squared length and the square of the first component, we deduce from the previous equivalence that $\xi\in\tilde{V}(\gam,\kappa)$ exactly when $\xi_1^2=\rho(\gam,\kappa)^2\chi(\gam,\kappa)^2$ and $\xi_2^2=\rho(\gam,\kappa)^2\tp{1-\chi(\gam,\kappa)^2}$. So indeed $V(\gam,\kappa)=\tilde{V}(\gam,\kappa)$ as claimed and the first item is now established.
		
		Let us now consider the second item. Firstly, one readily verifies (as a consequence of the first item) that $\m{span}\tcb{\upvarphi_+,\upvarphi_-}\subseteq\m{ker}Q(\gam,\kappa,\cdot)$. To show the opposite inclusion, we shall split into two cases. Consider first that $(\gam,\kappa)\in\mathfrak{E}\cap\tp{\pd\mathfrak{E}}$, in other words $0<\kappa-\gam=\f{4\mu}{\sig}\gam\tp{\gam^2-1}$. In this case we have that $\chi(\gam,\kappa)=1$ and so $V(\gam,\kappa)=\tcb{(-\rho(\gam,\kappa),0),(\rho(\gam,\kappa),0)}$. Since $\mathfrak{l}(\gam,\kappa)=\f{1}{\rho(\gam,\kappa)}\tp{1,1}$ we can use Fourier reconstruction paired with the first item to deduce that if $\eta\in\m{ker}Q(\gam,\kappa,\cdot)\subset\bf{X}^{\m{surf}}_L$ then 
		\begin{equation}
			\eta(x_1,x_2)=\f{1}{\sqrt{L_1L_2}}\tp{\mathscr{F}[\eta](-\rho(\gam,\kappa),0)e^{-2\pi\ii x_1/L_1}+\mathscr{F}[\eta](\rho(\gam,\kappa),0)e^{2\pi\ii x_1/L_1}}.
		\end{equation}
		$\eta$ is $\R$-valued and so $\mathscr{F}[\eta](-\rho(\gam,\kappa),0)=\Bar{\mathscr{F}[\eta](\rho(\gam,\kappa),0)}$ and the above expression simplifies to a linear combination of $\upvarphi_+$ and $\upvarphi_-$.
		
		The consideration of the case $0<\kappa-\gam<\f{4\mu}{\sig}\gam\tp{\gam^2-1}$ follows a similar argument. Now $\chi(\gam,\kappa)\in(0,1)$ and so $V(\gam,\kappa)$ consists of 4 points and Fourier reconstruction gives:
		\begin{equation}
			\eta(x_1,x_2)=\f{1}{\sqrt{L_1L_2}}\sum_{\sig_1,\sig_2\in\tcb{-1,1}}\mathscr{F}[\eta](\sig_1/L_1,\sig_2/L_2)e^{2\pi\ii\sig_1x_1/L_1}e^{2\pi\ii\sig_2x_2/L_2}.
		\end{equation}
		$\eta$ is again $\R$-valued, which means $\mathscr{F}[\eta](-\xi)=\Bar{\mathscr{F}[\eta](\xi)}$, and so the expression above simplifies to
		\begin{multline}
			\eta(x_1,x_2)=A_{0}\cos(2\pi x_1/L_1)\cos(2\pi x_2/L_2)+A_{1}\sin(2\pi x_1/L_1)\cos(2\pi x_2/L_2)\\+A_{2}\cos(2\pi x_1/L_1)\sin(2\pi x_2/L_2)+A_{3}\sin(2\pi x_1/L_1)\sin(2\pi x_2/L_2),
		\end{multline}
		for some coefficients $A_0,A_1,A_2,A_3\in\R$. Now we invoke that $\eta$ is an even function in the second variable, which necessitates that $A_2=A_3=0$ and hence $\eta$ is a linear combination of $\upvarphi_+$ and $\upvarphi_-$. This completes the proof.
	\end{proof}
	
	Now that we have established an explicit representation for the kernel of $Q$, we would like to return to the study of $P$. In what follows the following projection operator notation will be used. Define, for $(\gam,\kappa)\in\mathfrak{E}$ and $L=\mathfrak{l}(\gam,\kappa)$, the linear map $\Uppi_{\gam,\kappa}:\bf{X}^{\m{surf}}_{L}\to\m{ker}Q(\gam,\kappa,\cdot)$ via 
	\begin{equation}\label{kernel projection operator}
		\Uppi_{\gam,\kappa}\eta=\upvarphi_+\int_{\T^2_L}\eta\upvarphi_++\upvarphi_-\int_{\T^2_L}\eta\upvarphi_-,
	\end{equation}
	where the functions $\upvarphi_{\pm}$ are as in the second item of Proposition~\ref{proposition on computation of the reduced kernel}. Note that since $\upvarphi_{\pm}$ are orthonormal for the $L^2(\T^2_L)$ scalar product, we have that $\Uppi_{\gam,\kappa}\circ\Uppi_{\gam,\kappa}=\Uppi_{\gam,\kappa}$.
	
	We now combine the previous two results to summarize the information we have on the kernel of $P$.
	\begin{thm}[Kernel synthesis]\label{theorem on kernel synthesis}
		Suppose that $(\gam,\kappa)\in\mathfrak{E}$ and that $L=\mathfrak{l}(\gam,\kappa)\in\tp{\R^+}^2$.
		\begin{enumerate}
			\item The subspace $\m{ker}P(\gam,\kappa,\cdot,\cdot)\subset\bf{X}_L$ is two dimensional and is spanned by the set $\tcb{R_{\gam,\kappa}\upvarphi_+,R_{\gam,\kappa}\upvarphi_-}$ with $\upvarphi_{\pm}$ as described in the second item of Proposition~\ref{proposition on computation of the reduced kernel} and $R_{\gam,\kappa}$ the operator from~\eqref{the kernel isomorphism operator}.
			\item The following direct sum decomposition holds:
			\begin{equation}\label{direct sum decomposition for the domain}
				\bf{X}_L=\m{ker}P(\gam,\kappa,\cdot,\cdot)\oplus\bf{Z}_L(\gam,\kappa)
			\end{equation}
			where
			\begin{equation}
				\bf{Z}_{L}(\gam,\kappa)=\tcb{(u,\eta)\in\bf{X}_L\;:\;\mathscr{F}[\eta](\xi)=0,\;\forall\;\xi\in V(\gam,\kappa)}.
			\end{equation}
		\end{enumerate}
	\end{thm}
	\begin{proof}
		The first item is immediate from the combination of Propositions~\ref{proposition on the reduction of kernels} and~\ref{proposition on computation of the reduced kernel}: the former result shows that $R_{\gam,\kappa}:\m{ker}Q(\gam,\kappa,\cdot)\to\m{ker}P(\gam,\kappa,\cdot,\cdot)$ is an isomorphism while the latter result shows that $\m{ker}Q(\gam,\kappa,\cdot)=\m{span}\tcb{\upvarphi_+,\upvarphi_-}$.
		
		Let us now focus on proving the second item. Suppose that $(u,\eta)\in\bf{X}_L$ and define $(u_0,\eta_0),(u_1,\eta_1)\in\bf{X}_L$ via $(u_0,\eta_0)=R_{\gam,\kappa}\Uppi_{\gam,\kappa}\eta$ (recall~\eqref{kernel projection operator}) and $(u_1,\eta_1)=(u,\eta)-(u_0,\eta_0)$. By construction we have that $(u,\eta)=(u_0,\eta_0)+(u_1,\eta_1)$. Since $\Uppi_{\gam,\kappa}\eta\in\m{ker}Q(\gam,\kappa,\cdot)$ Proposition~\ref{proposition on the reduction of kernels} then assures us that $(u_0,\eta_0)\in\m{ker}P(\gam,\kappa,\cdot,\cdot)$. To verify that $(u_1,\eta_1)\in\bf{Z}_L(\gam,\kappa)$, we need that $\mathscr{F}[\eta_1]$ vanishes on $V(\gam,\kappa)$. Since $\eta_1=(I-\Uppi_{\gam,\kappa})\eta$ and $\Uppi_{\gam,\kappa}$ is a projection operator, we find that $\Uppi_{\gam,\kappa}\eta_1=0$. On the other hand, we can write $\eta_1=\eta_{1,0}+\eta_{1,1}$ with $\eta_{1,0}=\mathscr{F}^{-1}[\mathds{1}_{V(\gam,\kappa)}\mathscr{F}[\eta_1]]$ and $\eta_{1,1}=\eta_1-\eta_{1,0}$. Since $\eta_{1,1}$ contains exclusively frequencies outside of $V(\gam,\kappa)$, we have $\Uppi_{\gam,\kappa}\eta_{1,1}=0$ as well and hence $\Uppi_{\gam,\kappa}\eta_{1,0}=0$. On the other hand $\eta_{1,0}\in\m{span}\tcb{\upvarphi_+,\upvarphi_-}$ since the proof of Proposition~\ref{proposition on computation of the reduced kernel} showed that $\m{span}\tcb{\upvarphi_+,\upvarphi_-}=\tcb{h\in\bf{X}_L^{\m{surf}}\;:\;\m{supp}\mathscr{F}[h]\subseteq V(\gam,\kappa)}$. So we deduce that $\eta_{1,0}=0$ and so the inclusion $(u_1,\eta_1)\in\bf{Z}_L(\gam,\kappa)$ follows.
		
		We have established that the subspaces on the right of~\eqref{direct sum decomposition for the domain} sum to $\bf{X}_L$. Let us now argue that the sum is direct. Suppose that $(u_0,\eta_0)\in\m{ker}P(\gam,\kappa,\cdot,\cdot)$ and $(u_1,\eta_1)\in\bf{Z}_L(\gam,\kappa)$ are such that $(u_0,\eta_0)+(u_1,\eta_1)=0$. Upon restricting our attention to the $\eta$ component we learn that $\eta_0+\eta_1=0$ with $\mathscr{F}[\eta_0]$ supported in $V(\gam,\kappa)$ (thanks to Propositions~\ref{proposition on the reduction of kernels} and~\ref{proposition on computation of the reduced kernel}) and $\mathscr{F}[\eta_1]$ supported in $\tp{\Z/L_1\times\Z/L_2}\setminus V(\gam,\kappa)$. So necessarily we have $\eta_0=\eta_1=0$. We also know that $(u_0,\eta_0)=R_{\gam,\kappa}\eta_0$ and so $u_0=0$ as well. At last we can conclude that $u_1=0$ and so the sum is indeed direct.
	\end{proof}
	
	%*_*_*()(())()*_*_*()(())()*_*_*()(())()*_*_*()(())()*_*_*()(())()*_*_*()(())()*_*_*()(())()*_*_*()(())()
	\subsection{Analysis of the range}\label{subsection on analysis of the range}
	%*_*_*()(())()*_*_*()(())()*_*_*()(())()*_*_*()(())()*_*_*()(())()*_*_*()(())()*_*_*()(())()*_*_*()(())()
	
	In this subsection we complement our previous analysis by now studying the ranges of the operators $P$ from~\eqref{the family of linearized operators P}. We shall find the following auxiliary linear operator quite useful: define $S_\gam:\Bar{\bf{Y}}_L\to H^3_\bullet(\T^2_L)$, for $\gam\in\R$, via
	\begin{equation}\label{the auxilary range operator}
		S_\gam(f,g)=\Delta^{-2}\tsb{\grad\cdot f-(1-\gam\pd_1-4\mu\Delta)g},\quad(f,g)\in\Bar{\bf{Y}}_L.
	\end{equation}
	Note that $S_\gam$ is also a parity preserving operator in the sense that $S_\gam:\bf{Y}_L\to\bf{X}_L^{\m{surf}}$. Let us now classify the image of the operators $P(\gam,\kappa,\cdot,\cdot)$, viewing them as mapping $\bf{X}_L\to\bf{Y}_L$.
	\begin{prop}[Computation of the range]\label{Prop on computation of the range}
		Suppose that $(\gam,\kappa)\in\mathfrak{E}$ and $L=\mathfrak{l}(\gam,\kappa)\in\tp{\R^+}^2$. The following hold.
		\begin{enumerate}
			\item For $(f,g)\in\bf{Y}_L$ we have that $(f,g)\in\m{ran}P(\gam,\kappa,\cdot,\cdot)$ if and only if $\m{supp}\mathscr{F}[S_\gam(f,g)]\subseteq\tp{\Z/L_1\times\Z/L_2}\setminus V(\gam,\kappa)$.
			\item The subspace $\m{ran}P(\gam,\kappa,\cdot,\cdot)$ is closed.
		\end{enumerate}
	\end{prop}
	\begin{proof}
		Let us begin by establishing the first item. Assume first that $(f,g)\in\m{ran}P(\gam,\kappa,\cdot,\cdot)$ so that there exists $(u,\eta)\in\bf{X}_L$ such that $P(\gam,\kappa,u,\eta)=(f,g)$. We therefore have the equations
		\begin{equation}\label{expanded P}
			-\gam\pd_1 u+u-\mu\grad\cdot\mathbb{S}u+\grad\tp{1-\sig\Delta}\eta=\kappa\eta e_1+f,\quad\grad\cdot u=\gam\pd_1\eta+g.
		\end{equation}
		The first of these can be decomposed into solenoidal and potential parts via the Leray projection operator $\mathbb{P}$ and the second can be written in terms of $(I-\mathbb{P})$:
		\begin{multline}\label{Leray a GoGo}
			(1-\gam\pd_1-\mu\Delta)\mathbb{P}u=\kappa\mathbb{P}(\eta e_1)+\mathbb{P}f,\quad (1-\gam\pd_1-4\mu\Delta)(I-\mathbb{P})u+\grad(1-\sig\Delta)\eta=\kappa(I-\mathbb{P})(\eta e_1)+(I-\mathbb{P})f,\\
			(I-\mathbb{P})u=\gam(I-\mathbb{P})(\eta e_1)+\Delta^{-1}\grad g.
		\end{multline}
		We then substitute the final identity in~\eqref{Leray a GoGo} into the penultimate one; this permits us to derive that the above equations are equivalent to
		\begin{equation}\label{the isolated eta equation}
			Q(\gam,\kappa,\eta)=S_\gam(f,g)
		\end{equation}
		and
		\begin{equation}\label{recovery of the velocity field}
			u=(1-\gam\pd_1-\mu\Delta)^{-1}\tp{\kappa\mathbb{P}(\eta e_1)+\mathbb{P}f}+\gam(I-\mathbb{P})(\eta e_1)+\Delta^{-1}\grad g.
		\end{equation}
		Recall that the operators $Q$ are defined in~\eqref{the family of auxilary linearized operators Q}. Equation~\eqref{recovery of the velocity field} tell us that the velocity $u$ is entirely determined from the data $(f,g)$ and the free surface $\eta$ while~\eqref{the isolated eta equation} is an isolated equation for the free surface in terms of the data.
		
		Upon taking the Fourier transform of identity~\eqref{the isolated eta equation} we find that for all $\xi\in\tp{\Z/L_1\times\Z/L_2}\setminus\tcb{0}$.
		\begin{equation}\label{the isolated eta equation in Fourier space}
			\tp{4\pi^2|\xi|^2}^{-2}q_{\gam,\kappa}(\xi)\mathscr{F}[\eta](\xi)=\mathscr{F}[S_\gam(f,g)](\xi)
		\end{equation}
		for the symbol
		\begin{equation}\label{the q symbol}
			q_{\gam,\kappa}(\xi)=2\pi\ii\xi_1\tp{\gam-\kappa+16\pi^2\mu\gam|\xi|^2}-4\pi^2\tp{|\xi|^2-\gam^2\xi_1^2+4\pi^2\sig|\xi|^4}.
		\end{equation}
		The proof of Proposition~\ref{proposition on computation of the reduced kernel} shows that
		\begin{equation}
			\tcb{\xi\in\tp{\Z/L_1\times\Z/L_2}\setminus\tcb{0}\;:\;q_{\gam,\kappa}(\xi)=0}=\tilde{V}(\gam,\kappa)=V(\gam,\kappa),
		\end{equation}
		where the middle set is given by~\eqref{the set of interesting frequencies}. Therefore, a necessary condition for~\eqref{the isolated eta equation in Fourier space} to hold is that $\mathscr{F}[S_\gam(f,g)]$ vanishes whenever $q_{\gam,\kappa}$ vanishes, which is the same as saying that $\mathscr{F}[S_\gam(f,g)]$ is supported on the complement of $V(\gam,\kappa)$. Thus we have established the necessary direction of the first item.
		
		Let us now focus on the sufficient direction. Suppose that $(f,g)\in\bf{Y}_L$ satisfy $\m{supp}\mathscr{F}[S_\gam(f,g)]\subseteq\tp{\Z/L_1\times\Z/L_2}\setminus V(\gam,\kappa)$. We propose that one can define $(u,\eta)$ via~\eqref{the isolated eta equation} and~\eqref{recovery of the velocity field} under this hypothesis. We shall define $\eta\in\bf{X}_L^{\m{surf}}$ via the following prescription of its Fourier modes $\xi\in\Z/L_1\times\Z/L_2$:
		\begin{equation}\label{the construction of eta for the linearized problem}
			\mathscr{F}[\eta](\xi)=\begin{cases}
				0&\text{if }\xi\in V(\gam,\kappa)\cup\tcb{0},\\
				\tp{q_{\gam,\kappa}(\xi)}^{-1}\tp{4\pi^2|\xi|^2}^2\mathscr{F}[S_\gam(f,g)](\xi)&\text{otherwise.}
			\end{cases}
		\end{equation}
		Note that if $\xi\notin V(\gam,\kappa)\cup\tcb{0}$ then $q_{\gam,\kappa}(\xi)\neq0$; additionally it holds that
		\begin{equation}\label{the large frequency limit}
			\f{\tp{4\pi^2|\xi|^2}^2}{q_{\gam,\kappa}(\xi)}\to-\f{1}{\sig}\quad\text{as}\quad|\xi|\to\infty
		\end{equation}
		and hence not only is $\eta\in\bf{X}_L^{\m{surf}}$ well-defined, we also have the estimate $\tnorm{\eta}_{\bf{X}_L^{\m{surf}}}\lesssim\tnorm{S_\gam(f,g)}_{\bf{X}_L^{\m{surf}}}\lesssim\tnorm{(f,g)}_{\bf{Y}_L}$ for an implicit constant depending only $\gam$, $\kappa$, $\mu$, and $\sig$. We then define $u$ in terms of $\eta$, $f$, and $g$ according to equation~\eqref{recovery of the velocity field}. It is then readily verified that $u\in\bf{X}_L^{\m{velo}}$ with $\tnorm{u}_{\bf{X}_L^{\m{velo}}}\lesssim\tnorm{\tp{f,g}}_{\bf{Y}_L}$. Let us now verify that $P(\gam,\kappa,u,\eta)=(f,g)$; by the reductions made at the beginning of the proof it is equivalent to show that~\eqref{the isolated eta equation} and~\eqref{recovery of the velocity field} are satisfied. The latter of this is automatic while the former of these follows from~\eqref{the construction of eta for the linearized problem} and the assumption that $\mathscr{F}[S_\gam(f,g)](\xi)=0$ for $\xi\in V(\gam,\kappa)\cup\tcb{0}$. Thus the first item's proof is now complete.
		
		The second item is easy now that we have the first. The range is a closed set since the first conclusion equates the range to the intersection of finitely many kernels of bounded linear maps on $\bf{Y}_L$; namely the functions $(f,g)\mapsto\mathscr{F}[S_\gam(f,g)](\xi)$ for $\xi\in V(\gam,\kappa)$.
	\end{proof}
	
	Next we shall consider certain complimented subspaces to the range of $P$ which are related to the partial derivatives of $P$ with respect to the $\gam$ and $\kappa$ parameters and its kernel. The following notation is set: the mappings $P_{\m{s}},P_{\m{t}}:\bf{X}_{L}\to\bf{Y}_L$ are defined via
	\begin{equation}\label{the operators Ps and Pt}
		P_{\m{s}}(u,\eta)=-\tp{\pd_1 u,\pd_1\eta},\quad P_{\m{t}}(u,\eta)=-(\eta e_1,0).
	\end{equation}
	Notice that
	\begin{equation}
		P(\gam,\kappa,u,\eta)=P(0,0,u,\eta)+\gam P_{\m{s}}(u,\eta)+\kappa P_{\m{t}}(u,\eta),
	\end{equation}
	for $(u,\eta)\in\bf{X}_L$ and $(\gam,\kappa)\in\R^2$. In particular we have $D_{1,1}P(\gam,\kappa,\cdot,\cdot)=P_{\m{s}}$ and $D_{1,2}P(\gam,\kappa,\cdot,\cdot)=P_{\m{t}}$, where $D_{1,1}$ and $D_{1,2}$ refer to the coordinate partial derivatives with respect to the $\R^2$ factor in the definition of $P$ (which is~\eqref{the family of linearized operators P}).
	
	\begin{lem}[Bases for the kernel of $Q$]\label{lemma on bases for the kernel of Q}
		Suppose that $\tp{\gam,\kappa}\in\mathfrak{E}$, $L=\mathfrak{l}(\gam,\kappa)\in\tp{\R^+}^2$. The following hold:
		\begin{enumerate}
			\item $S_\gam P_{\m{s}},S_\gam P_{\m{t}}:\m{ker}P(\gam,\kappa,\cdot,\cdot)\to\m{ker}Q(\gam,\kappa,\cdot)$, where the operator $S_\gam$ is defined in~\eqref{the auxilary range operator} and the operators $P_{\m{s}}$ and $P_{\m{t}}$ are given in~\eqref{the operators Ps and Pt}.
			\item For all $(u,\eta)\in\m{ker}P(\gam,\kappa,\cdot,\cdot)\setminus\tcb{0}$ the linear map
			\begin{equation}\label{the linear map that gives and interesting basis for q}
				\R^2\ni(\al,\be)\mapsto\al S_\gam P_{\m{s}}(u,\eta)+\be S_{\gam}P_{\m{t}}(u,\eta)\in\m{ker}Q(\gam,\kappa,\cdot)
			\end{equation}
			is invertible.
		\end{enumerate}
	\end{lem}
	\begin{proof}
		We begin the proof by computing the operators $S_\gam P_{\m{s}}$ and $S_\gam P_{\m{t}}$ on $\m{ker}P(\gam,\kappa,\cdot,\cdot)$, so let $(u,\eta)$ belong to this latter set. Thanks to Proposition~\ref{proposition on the reduction of kernels} we have $(u,\eta)=R_{\gam,\kappa}\eta$ with $\eta\in\m{ker}Q(\gam,\kappa,\cdot)$. So we can directly compute that $\Delta^2 S_\gam P_{\m{s}}(u,\eta)=-\grad\cdot(\pd_1 u)+(1-\gam\pd_1-4\mu\Delta)\pd_1\eta=-2\gam\pd_1^2\eta+(1-4\mu\Delta)\pd_1\eta$ and $\Delta^2 S_\gam P_{\m{t}}\eta=-\grad\cdot(\eta e_1)=-\pd_1\eta$ and hence (since $\eta$ is frequency supported $V(\gam,\kappa)$)
		\begin{multline}\label{the expressions for the partial operators on the kernel}
			S_\gam P_{\m{s}}(u,\eta)=\f{1}{(4\pi^2)^2\rho(\gam,\kappa)^4}\tp{8\gam\pi^2\rho(\gam,\kappa)^2\chi(\gam,\kappa)^2\eta+(1+16\pi^2\mu\rho(\gam,\kappa)^2)\pd_1\eta},\\S_\gam P_{\m{t}}\eta=-\f{1}{(4\pi^2)^2\rho(\gam,\kappa)^4}\pd_1\eta.
		\end{multline}
		Notice that the operations $\pd_1$ and $\pd_1^2$ preserve the subspace $\m{ker}Q(\gam,\kappa,\cdot)$ (as can be readily checked on the basis $\tcb{\upvarphi_+,\upvarphi_-}$) and hence we can deduce from the expressions~\eqref{the expressions for the partial operators on the kernel} the claimed mapping properties of the first item.
		
		We now prove the second item by checking a matrix representation for the linear map~\eqref{the linear map that gives and interesting basis for q}. The domain $\R^2$ is equipped with the standard basis while the codomain $\m{ker}Q(\gam,\kappa,\cdot)$ shall be equipped with the basis $\tcb{\eta,\pd_1\eta}$. That this latter set is indeed a basis, we shall check now. Supposing that $(u,\eta)\in\m{ker}P(\gam,\kappa,\cdot,\cdot)\setminus\tcb{0}$ leads to $\eta\neq0$ thanks to the previously mentioned identity $(u,\eta)=R_{\gam,\kappa}\eta$. This $\eta$ belongs to $\m{ker}Q(\gam,\kappa,\cdot)$ and (by Proposition~\ref{proposition on computation of the reduced kernel}) has frequency support in the set $V(\gam,\kappa)$. As this set does not intersect $\tcb{\xi\in\Z/L_1\times\Z/L_2\;:\;\xi_1=0}$  we also have $\pd_1\eta\neq0$. We also know that $\m{ker}Q(\gam,\kappa,\cdot)$ is two dimensional and can be equipped with the inner product $\tbr{\cdot,\cdot}$ that is given by the $L^2(\T^2_L)$-inner product of functions. As $\tbr{\eta,\pd_1\eta}=0$ (by the divergence theorem) we necessarily have that $\eta$ and $\pd_1\eta$ are linearly independent and hence spanning. Written in these aforementioned bases, the matrix $M\in\R^{2\times 2}$ corresponding to the linear map~\eqref{the linear map that gives and interesting basis for q} (which is computed from the expressions~\eqref{the expressions for the partial operators on the kernel}) is given by
		\begin{equation}
			M=\f{1}{\tp{4\pi^2}^2\rho(\gam,\kappa)^4}\bpm8\gam\pi^2\rho(\gam,\kappa)^2\chi(\gam,\kappa)^2&0\\1+16\pi^2\mu\rho(\gam,\kappa)^2&-1\epm.
		\end{equation}
		As $\rho(\gam,\kappa),\chi(\gam,\kappa)\in\R^+$ whenever $(\gam,\kappa)\in\mathfrak{E}$, we see that $\det M\neq 0$ and the second item now follows.
	\end{proof}
	
	\begin{prop}[Complements of the range]\label{prop on complements of the range}
		Let $(\gam,\kappa)\in\mathfrak{E}$, $L=\mathfrak{l}(\gam,\kappa)$. Then for all $(u,\eta)\in\m{ker}P(\gam,\kappa,\cdot,\cdot)\setminus\tcb{0}$ we have
		\begin{equation}\label{the direct sum decomposition of the codomain}
			\m{ran}P(\gam,\kappa,\cdot,\cdot)\oplus\m{span}P_{\m{s}}(u,\eta)\oplus\m{span}P_{\m{t}}(u,\eta)=\bf{Y}_L.
		\end{equation}
		In particular, $\m{codim}\;\m{ran}P(\gam,\kappa,\cdot,\cdot)=2$.
	\end{prop}
	\begin{proof}
		Let us establish first that the sum of the three subspaces on the left of~\eqref{the direct sum decomposition of the codomain} is $\bf{Y}_L$. Let $(f,g)\in\bf{Y}_L$. Set $h=\Uppi_{\gam,\kappa}S_\gam(f,g)\in\m{ker}Q(\gam,\kappa)$ (recall that the projection operator $\Uppi_{\gam,\kappa}$ is defined in equation~\eqref{kernel projection operator} while the operator $S_\gam$ is from~\eqref{the auxilary range operator}). Thanks to the second item of Lemma~\ref{lemma on bases for the kernel of Q} we are assured the existence of $\tp{\al,\be}\in\R^2$ such that $\al S_\gam P_{\m{s}}(u,\eta)+\be S_\gam P_{\m{t}}(u,\eta)=h$. We set $(\tilde{f},\tilde{g})=(f,g)-\al P_{\m{s}}(u,\eta)-\be P_{\m{t}}(u,\eta)\in\bf{Y}_L$ and claim that $(\tilde{f},\tilde{g})\in\m{ran}P(\gam,\kappa,\cdot,\cdot)$. To check this we shall use the first item of Proposition~\ref{Prop on computation of the range}. By construction we have that
		\begin{equation}
			S_\gam(\tilde{f},\tilde{g})=S_\gam(f,g)-\al S_\gam P_{\m{s}}(u,\eta)-\be S_\gam P_{\m{t}}(u,\eta)=(I-\Uppi_{\gam,\kappa})S_\gam(f,g)\in\bf{X}_{L}^{\m{surf}}.
		\end{equation}
		Hence $\Uppi_{\gam,\kappa}S_\gam(\tilde{f},\tilde{g})=0$. Now by arguing as in the proof of Theorem~\ref{theorem on kernel synthesis}, we deduce $\m{supp}\mathscr{F}[S_\gam(\tilde{f},\tilde{g})]\subseteq(\Z/L_1\times\Z/L_2)\setminus V(\gam,\kappa)$. So indeed $(\tilde{f},\tilde{g})\in\m{ran}P(\gam,\kappa,\cdot,\cdot)$.
		
		It remains to show that the sum of subspaces $\m{ran}P(\gam,\kappa,\cdot,\cdot)+\m{span}P_{\m{s}}(u,\eta)+\m{span}P_{\m{t}}(u,\eta)=\bf{Y}_L$ is direct. So suppose that $(f_0,g_0)\in\m{ran}P(\gam,\kappa,\cdot,\cdot)$ and $(\al_0,\be_0)\in\R^2$ are such that
		\begin{equation}\label{the directness checking equation}
			(f_0,g_0)+\al_0P_{\m{s}}(u,\eta)+\be_0P_{\m{t}}(u,\eta)=0.
		\end{equation}
		We can apply $\Uppi_{\gam,\kappa}S_\gam$ to~\eqref{the directness checking equation}; the contribution by $(f_0,g_0)$ vanishes thanks to the first item of Proposition~\ref{Prop on computation of the range} and so we are left with $\Uppi_{\gam,\kappa}[\al_0 S_\gam P_{\m{s}}(u,\eta)+\be_0S_{\gam}P_{\m{t}}(u,\eta)]=0$. Thanks to the first item of Lemma~\ref{lemma on bases for the kernel of Q} we then learn that $\al_0 S_\gam P_{\m{s}}(u,\eta)+\be_0S_{\gam}P_{\m{t}}(u,\eta)=0$ and then the second item of the same result implies that $\al_0=\be_0=0$. Upon returning to~\eqref{the directness checking equation} we deduce that $(f_0,g_0)=0$. So directness is shown and the proof is complete.
	\end{proof}
	
	%*_*_*()(())()*_*_*()(())()*_*_*()(())()*_*_*()(())()*_*_*()(())()*_*_*()(())()*_*_*()(())()*_*_*()(())()
	\subsection{Synthesis of linear analysis}\label{subsection on synthesis of the linear analysis}
	%*_*_*()(())()*_*_*()(())()*_*_*()(())()*_*_*()(())()*_*_*()(())()*_*_*()(())()*_*_*()(())()*_*_*()(())()
	
	In this final subsection on linear analysis we tie together our results of Subsections~\ref{subsection on analysis of the kernel} and~\ref{subsection on analysis of the range} in the following theorem.
	
	\begin{thm}[On the operator $P$, I]\label{theorem on the operator P}
		Let $(\gam,\kappa)\in\mathfrak{E}$ and $L=\mathfrak{l}(\gam,\kappa)\in\tp{\R^+}^2$. Regarding the operator $P(\gam,\kappa,\cdot,\cdot):\bf{X}_L\to\bf{Y}_L$ defined in~\eqref{the family of linearized operators P}, the following hold.
		\begin{enumerate}
			\item $\m{dim}\;\m{ker}P(\gam,\kappa,\cdot,\cdot)=2$.
			\item The range of $P(\gam,\kappa,\cdot,\cdot)$ is closed and $\m{codim}\;\m{ran}P(\gam,\kappa,\cdot,\cdot)=2$.
			\item For every $X\in\m{ker}P(\gam,\kappa,\cdot,\cdot)\setminus\tcb{0}$ it holds that
			\begin{equation}
				\bf{Y}_L=\m{ran}P(\gam,\kappa,\cdot,\cdot)\oplus\m{span}D_{1,1}P(\gam,\kappa,\cdot,\cdot)X\oplus\m{span}D_{1,2}P(\gam,\kappa,\cdot,\cdot)X.
			\end{equation}
		\end{enumerate}
	\end{thm}
	\begin{proof}
		The first item is Theorem~\ref{theorem on kernel synthesis}. The second and third items are Propositions~\ref{Prop on computation of the range} and~\ref{prop on complements of the range}.
	\end{proof}
	
	We now give a complementary result to Theorem~\ref{theorem on the operator P} which essentially says the set $\mathfrak{E}\cup\tp{-\mathfrak{E}}$ contains exactly the interesting parameter values for the linearized operator. Note that the next result holds in the larger domain and codomain spaces of~\eqref{larger than life spaces}, i.e. there is no need for second variable parity assumptions.
	
	\begin{thm}[On the operator P, II]\label{second thm on the operator P}
		Suppose that $(\gam,\kappa)\in\R^2$ and let $L\in\tp{\R^+}^2$. If $(\gam,\kappa)\not\in\mathfrak{E}\cup\tp{-\mathfrak{E}}$, then the operator $P(\gam,\kappa,\cdot,\cdot):\Bar{\bf{X}}_L\to\Bar{\bf{Y}}_L$ is an isomorphism.
	\end{thm}
	\begin{proof}
		By following the reductions in the proof of Proposition~\ref{Prop on computation of the range}, which do not require parity assumptions, we find that for $(u,\eta)\in\Bar{\bf{X}}_L$ and $(f,g)\in\Bar{\bf{Y}}_L$ the equation $P(\gam,\kappa,u,\eta)=(f,g)$ is equivalent to the satisfaction of equations~\eqref{the isolated eta equation} and~\eqref{recovery of the velocity field}. The latter equation reads $u$ as a bounded linear function of the data and $\eta$, so injectivity and surjectivity of $P(\gam,\kappa,\cdot,\cdot)$ follow as soon as we verify that $Q(\gam,\kappa,\cdot):H^3_\bullet(\T^2_L)\to H^3_\bullet(\T^2_L)$ is an isomorphism. This linear map corresponds to the Fourier multiplication operator $\Z/L_1\times\Z/L_2\ni\xi\mapsto\f{q_{\gam,\kappa}(\xi)}{(4\pi^2|\xi|^2)^2}\in\C$ (with the definition of $q_{\gam,\kappa}$ given in~\eqref{the q symbol}) and so we are tasked with inversion of its symbol on nonzero frequencies. Note that the limit~\eqref{the large frequency limit} shows that there is no issue for frequencies outside of a large enough bounded region. Therefore, we need only verify that $q_{\gam,\kappa}$ has no zeros in the set $\tp{\Z/L_1\times\Z/L_2}\setminus\tcb{0}$.
		
		Let  $Z(\gam,\kappa)=\tcb{\xi\in\tp{\Z/L_1\times\Z/L_2}\setminus\tcb{0}\;:\;q_{\gam,\kappa}(\xi)=0}$. By arguing as in the proof of Proposition~\ref{proposition on computation of the reduced kernel} we readily see that the set $Z(\gam,\kappa)$ does not intersect the line $\tcb{\xi\;:\;\xi_1=0}$ and so an equivalent description is
		\begin{equation}\label{the sets Z}
			Z(\gam,\kappa)=\bcb{\xi\in\tp{\Z/L_1\times\Z/L_2}\setminus\tcb{0}\;:\;\bpm\gam-\kappa+16\pi^2\mu\gam|\xi|^2\\|\xi|^2-\gam^2\xi_1^2+4\pi^2\sig|\xi|^4\epm=0}.
		\end{equation}
		We first show that if $\gam=0$ or $\kappa=0$, then $Z(\gam,\kappa)=\es$. If $\gam=0$ then the second equation in~\eqref{the sets Z} leads to $Z(\gam,\kappa)=\es$. If $\kappa=0$ but $\gam\neq0$, then the first equation in~\eqref{the sets Z} leads to $Z(\gam,\kappa)=\es$.
		
		We next claim that $Z(\gam,\kappa)=\es$ if $\gam$ and $\kappa$ are both non zero and have opposite signs. If there were to exist $\xi\in Z(\gam,\kappa)$ in this case, then the first equation in~\eqref{the sets Z} would imply that $1-\kappa/\gam+16\pi^2\mu|\xi|^2=0$ and so the opposite sign condition forces $|\xi|^2<0$ which is a contradiction.
		
		So $Z(\gam,\kappa)$ is only possibly nonempty if $\gam$ and $\kappa$ have the same sign. It is obvious that $Z(\gam,\kappa)=Z(-\gam,-\kappa)$ and so it suffices to consider what happens if both $\gam$ and $\kappa$ are positive. In this case we can argue as in the proof of Proposition~\ref{proposition on computation of the reduced kernel} to obtain that that $\xi\in Z(\gam,\kappa)$ if and only if the identities
		\begin{equation}\label{intersection of circle with figure 8}
			|\xi|^2=\f{\kappa/\gam-1}{16\pi^2\mu}\quad\text{and}\quad\xi_1^2=\f{1}{\gam^2}|\xi|^2\tp{1+4\pi^2\sig|\xi|^2}
		\end{equation}
		are satisfied. A necessary condition for the satisfaction of~\eqref{intersection of circle with figure 8} is that $|\xi|^2>0$ and $\xi_1^2\le|\xi|^2$ which implies $\gam>1$, $\kappa>\gam$, and $\kappa<\gam+\f{4\mu}{\sig}\gam\tp{\gam^2-1}$. In other words $Z(\gam,\kappa)\neq\es$ implies that $(\gam,\kappa)\in\mathfrak{E}$.
		
		It is thus established that if $(\gam,\kappa)\not\in\mathfrak{E}\cup\tp{-\mathfrak{E}}$, then $Z(\gam,\kappa)=\es$ and therefore the operator $Q(\gam,\kappa,\cdot):H^3_\bullet(\T^2_L)\to H^3_\bullet(\T^2_L)$ is an isomorphism and so too is $P(\gam,\kappa,\cdot,\cdot):\Bar{\bf{X}}_L\to\Bar{\bf{Y}}_L$.
	\end{proof}
	
	%*_*_*()(())()*_*_*()(())()*_*_*()(())()*_*_*()(())()*_*_*()(())()*_*_*()(())()*_*_*()(())()*_*_*()(())()
	\section{Nonlinear Analysis}\label{section on nonlinear analysis}
	%*_*_*()(())()*_*_*()(())()*_*_*()(())()*_*_*()(())()*_*_*()(())()*_*_*()(())()*_*_*()(())()*_*_*()(())()
	
	The purpose of this section of the document is to derive an abstract bifurcation theorem and then to use this tool to construct and classify all small amplitude solutions to~\eqref{traveling form of the equations} in our functional framework.
	
	%*_*_*()(())()*_*_*()(())()*_*_*()(())()*_*_*()(())()*_*_*()(())()*_*_*()(())()*_*_*()(())()*_*_*()(())()
	\subsection{Abstract bifurcation}\label{subsection on abstract bifurcation}
	%*_*_*()(())()*_*_*()(())()*_*_*()(())()*_*_*()(())()*_*_*()(())()*_*_*()(())()*_*_*()(())()*_*_*()(())()
	
	The derivation of our main tool which allows us to find all small amplitude gravity-capillary roll wave solutions to the inclined viscous shallow water system is the content of this subsection. The bifurcation theorem presented here for operators with multiple parameters, which will certainly be no surprise to the expert reader, is based off of Theorem I.19.6 in Kielh\"ofer~\cite{MR2859263}. The main difference is that our result adds a classification of the totality of small solutions by appending an argument mimicking that of Lemma 1.12 in Crandall and Rabinowitz~\cite{MR0288640}.
	\begin{thm}[Multiparameter bifurcation]\label{theorem on multiparameter bifurcation}
		Suppose that $X$ and $Y$ are real Banach spaces, $n\in\N^+$, $U\subset \R^n\times X$ is open, $f:U\to Y$ is smooth, and $(\lambda_\star,0)\in U$ are such that:
		\begin{enumerate}
			\item  $f(\lambda_\star+\lambda,0)=0$ whenever $(\lambda_\star+\lambda,0)\in U$.
			\item The linear map $D_2f(\lambda_\star,0):X\to Y$ has closed range and satisfies $\m{dim}E_\star=n$ and $\m{codim}F_\star=n$ where $E_\star=\m{ker}D_2f(\lambda_\star,0)\subset X$ and $F_\star=\m{ran}D_2f(\lambda_\star,0)\subset Y$. Here $D_2f(\lambda_\star,0)$ is the derivative with respect to the $X$ factor.
			\item The open set
			\begin{equation}\label{the direct sum set}
				A_\star=\tcb{v\in E_\star\setminus\tcb{0}\;:\;Y=F_\star\oplus\m{span}D_{1,1}D_2f(\lambda_\star,0)v\oplus\cdots\oplus\m{span} D_{1,n}D_2f(\lambda_\star,0)v}.
			\end{equation}
			is nonempty. Here $D_{1}=\tp{D_{1,1},\dots,D_{1,n}}$ is the derivative with respect to the $\R^n$ factor broken into the coordinate partials.
		\end{enumerate}
		If $Z\subset X$ is any subspace such that $E_\star\oplus Z=X$ and $\es\neq K\subset A_\star$ is compact, then the following hold.
		\begin{enumerate}
			\item There exists $\ep,r\in\R^+$ and smooth functions $\Lambda:[-\ep,\ep]\times K\to B_{\R^n}(0,r)$, $W:[-\ep,\ep]\times K\to Z\cap B_{X}(0,r)$ satisfying for all $(\al,v)\in[-\ep,\ep]\times K$:
			\begin{equation}\label{what are these functions really doing}
				\begin{cases}
					\Lambda(0,v)=0,\\
					W(0,v)=0,
				\end{cases}\quad\text{and}\quad f(\lambda_\star+\Lambda(\al,v),\al\tp{v+W(\al,v)})=0.
			\end{equation}
			\item If $(\lambda,z)\in B_{\R^n}(0,r)\times\tp{Z\cap B_X(0,r)}$ are such that for some $(\al,v)\in[-\ep,\ep]\times K$ with $\al\neq0$ we have that $f(\lambda_\star+\lambda,\al(v+z))=0$, then we have that $\lambda=\Lambda(\al,v)$ and $z=W(\al,v)$.
			\item Assume additionally that $K\supseteq\tcb{v\in E_\star\;:\;\tnorm{v}_X=1}$. Then there exists an open set $(\lambda_\star,0)\in U_1\subseteq U$ with the property that
			\begin{equation}\label{totality of local solutions}
				\tp{f^{-1}\tcb{0}}\cap U_1=\tsb{\tcb{(\lambda,0)\;:\;\lambda\in\R}\cup\tcb{(\lambda_\star+\Lambda(\al,v),\al\tp{v+W(\al,v)})\;:\;(\al,v)\in[-\ep,\ep]\times K}}\cap U_1.
			\end{equation}
		\end{enumerate}
	\end{thm}
	\begin{proof}
		Let 
		\begin{equation}
			V=\tcb{\tp{\al,v,\lambda,z}\in\R\times A_\star\times\R^n\times Z\;:\;(\lambda_\star+\lambda,\al(v+z))\in U)}.
		\end{equation}
		Observe that $\tcb{0}\times A_\star\times\tcb{0}\times\tcb{0}\subset V$. Consider the mapping $ F:V\subset \R\times A_\star\times \R^n\times Z\to Y$ defined via
		\begin{equation}
			F(\al,v,\lambda,z)=\begin{cases}
				\al^{-1}f(\lambda_\star+\lambda,\al(v+z))&\text{if }\al\neq0,\\
				D_2f(\lambda_\star+\lambda,0)(v+z)&\text{if }\al=0.
			\end{cases}
		\end{equation}
		Notice that the first hypothesis on $f$ permits us to Taylor expand for $\al\neq 0$ and write
		\begin{multline}
			\al^{-1}f(\lambda_\star+\lambda,\al(v+z))=\al^{-1}\tp{f(\lambda_\star+\lambda,\al(v+z))-f(\lambda_\star+\lambda,0)}\\=D_2f(\lambda_\star+\lambda,0)(v+z)+\al\int_0^1(1-\tau)D_2^2f(\lambda_\star+\lambda,\tau\al(v+z))(v+z)^{\otimes 2}\;\m{d}\tau.
		\end{multline}
		From this one then readily deduces that $F$ is smooth. We also observe that for any $v\in A_\star\subset E_\star$ it holds that $F(0,v,0,0)=0$.
		
		Now consider the domain of $F$ as grouped into the product of the factors $X_0=\R\times A_\star$ and $X_1=\R^n\times Z$ so that $V\subset X_0\times X_1$. We would like to establish that for $v\in A_\star$ the derivative $D_2F(0,v,0,0):X_1\to Y$ is an isomorphism of Banach spaces. We compute that for $(\Bar{\lambda},\Bar{z})\in X_1$ it holds
		\begin{equation}\label{the derivative of the auxilary mapping}
			D_2F(0,v,0,0)(\Bar{\lambda},\Bar{z})=D_2f(\lambda_\star,0)\Bar{z}+\sum_{j=1}^n\Bar{\lambda}_jD_{1,j}D_2f(\lambda_\star,0)v.
		\end{equation}
		Our hypotheses ensure that this map is an isomorphism. Indeed, by the definition of $A_\star$ and the fact that $v\in A_\star$, $D_2f(\lambda_\star,0)\Bar{z}\in F_\star$ it follows that for any $y\in Y$ the identity
		\begin{equation}
			D_2f(\lambda_\star,0)\Bar{z}+\sum_{j=1}^n\Bar{\lambda}_jD_{1,j}D_2f(\lambda_\star)v=y
		\end{equation}
		is equivalent to
		\begin{equation}\label{direct sum decomposition}
			D_2f(\lambda_\star,0)\Bar{z}=y_0,\;\forall\;j\in\tcb{1,\dots,n}\;\Bar{\lambda}_jD_{1,j}D_2f(\lambda_\star,0)v=y_j
		\end{equation}
		where $y=y_0+y_1+\cdots+y_n$ is the unique decomposition with $y_0\in F_\star$ and for $j\in\tcb{1,\dots,n}$ $y_j\in\m{span}D_{1,j}D_2f(\lambda_\star,0)v$. So if $y=0$, then $\Bar{\lambda}=0$ and $z\in Z\cap E_\star=\tcb{0}$, hence $D_2F(0,v,0,0)$ is injective. On the other hand, for any value of $y$ we can find $\Bar{z}\in Z$ and $\Bar{\lambda}\in\R^n$ such that equation~\eqref{direct sum decomposition} is satisfied; therefore surjectivity is too established.
		
		The hypotheses of the implicit function theorem are satisfied for the map $F$ at the point $(0,v,0,0)$. We are therefore granted radii $r^0_v,r^1_v\in\R^+$ with the property that $B_{X_0}((0,v),r^0_v)\times B_{X_1}((0,0),r^1_v)\subset V$ and we are also granted a unique function with $\upiota_v(0,v)=0$,
		\begin{equation}
			\upiota_v:B_{X_0}((0,v),r_v^0)\to B_{X_1}((0,0),r_v^1),\quad F(\al,w,\upiota_v(\al,w))=0,\;\forall\;(\al,w)\in B_{X_0}((0,v),r_v^0).
		\end{equation}
		Moreover the functions $\upiota_v$ are smooth.
		
		We now would like to glue together these maps. Notice that if $v,\Bar{v}\in A_\star$ are such that $B_{X_0}((0,v),r^0_v)\cap B_{X_0}((0,\Bar{v}),r_{\Bar{v}}^0)\neq\es$ then, by uniqueness, we must have that $\upiota_v=\upiota_{\Bar{v}}$ on this intersection. This allows us to define the smooth map
		\begin{equation}
			\upiota:\bigcup_{v\in A_\star} B_{X_0}((0,v),r_v^0)\to X_1,\quad\upiota=\upiota_v\text{ on }B_{X_0}((0,v),r^0_v).
		\end{equation}
		
		Now consider a compact subset $K\subset A_\star$. The family of balls $\tcb{B_{X_0}((0,v),r_v^0)}_{v\in K}$ form a cover of $\tcb{0}\times K\subset X_0$ and so there is a finite subcover, say indexed by $\tcb{v_1,\dots,v_m}\subset K$ for some $m\in\N^+$. Let us set $\rho_K=\f12\min\tcb{r^1_{v_1},\dots, r^1_{v_m}}\in\R^+$. The set
		\begin{equation}
			\upiota^{-1}(B_{X_1}((0,0),\rho_K))\subset \bigcup_{v\in A_\star} B_{X_0}((0,v),r_v^0) 
		\end{equation}
		is open and $\tcb{0}\times K\subset\upiota^{-1}\tp{B_{X_1}((0,0),\rho_K)}$. Thanks to compactness again, there exists a $\del_K\in\R^+$ with the property that $[-\del_K,\del_K]\times K\subset\upiota^{-1}\tp{B_{X_1}((0,0),\rho_K)}$. In other words, we can view $\upiota:[-\del_K,\del_K]\times K\to B_{X_1}((0,0),\rho_K)$ as a smooth map with the property that $\upiota(0,v)=0$ for all $v\in K$ and $F(\al,v,\upiota(\al,w))=0$ for all $(\al,w)\in[-\del_K,\del_K]\times K$. This restriction of $\upiota$ is the unique function with these two properties. To obtain the first and second conclusions of the theorem we need only to take $\Lambda$ and $W$ to be the components of the map $\upiota$.
		
		We now turn our attention to the third conclusion of the theorem, whose proof will be split into the series of sub-claims. The first sub-claim that will be established is as follows: There exists a constant $C\in\R^+$ such that for all $v\in K$ and all $(\lambda,z)\in\R^n\times Z$ we have
		\begin{equation}\label{the equivalence of lengths}
			C^{-1}\tp{\tnorm{z}_X+|\lambda|}\le\bnorm{D_2f(\lambda_\star,0)z+\sum_{j=1}^n\lambda_jD_{1,j}D_2f(\lambda_\star,0)v}_Y\le C\tp{\tnorm{z}_X+|\lambda|}.
		\end{equation}
		Notice that the map in the center of inequality~\eqref{the equivalence of lengths} is $D_2F(0,v,0,0)(\lambda,z)$ (see~\eqref{the derivative of the auxilary mapping}); earlier in the proof we have established that $D_2F(0,v,0,0)$ is an isomorphism $\R^n\times Z\to Y$. Therefore there exists a constant $C_v\in\R^+$ such that~\eqref{the equivalence of lengths} holds with $C$ replaced by $C_v$ at a fixed $v\in K$ and for all $(\lambda,z)\in\R^n\times Z$. By a simple absorption argument, we find that there actually exists $\upepsilon_v\in\R^+$ such that for all $\tilde{v}\in B(v,\upepsilon_v)\cap K$ and all $(\lambda,z)\in\R^n\times Z$ it holds that
		\begin{equation}
			\tp{2C_{v}}^{-1}\tp{\tnorm{z}+|\lambda|}\le\tnorm{D_2F(0,\tilde{v},0,0)(\lambda,z)}_Y\le 2C_v\tp{\tnorm{z}_X+|\lambda|}.
		\end{equation}
		The collection of open balls $\tcb{B(v,\upepsilon_v)}_{v\in K}$ is cover for the compact set $K$ and so we may extract a finite collection $\tcb{v_\ell}_{\ell=1}^p\subseteq K$ such that $K\subset\bigcup_{\ell=1}^pB(v_\ell,\upepsilon_{v_\ell})$; hence upon setting $C=2\max_{1\le \ell\le p}C_{v_\ell}$ we deduce the desired equivalence of~\eqref{the equivalence of lengths}.
		
		The second sub-claim is that there exists an open set $(\lambda_\star,0)\in U_1\subseteq U$ and $c\in\R^+$ such that for all $(\al,v)\in[-\ep,\ep]\times K$ and $(\lambda,z)\in \R^n\times Z$ satisfying $(\lambda_\star+\lambda,\al v+z)\in U_1$ and $f(\lambda_\star+\lambda,\al v+z)=0$ we have that
		\begin{equation}\label{the inequality of the second subclaim}
			|\al\lambda|+\tnorm{z}_X\le c|\al|^2\quad\text{and}\quad|\al|\le r/2c,
		\end{equation}
		where $r\in\R^+$ is the radius granted by the first conclusion of the theorem.
		
		To prove this claim we fix $\kappa\in\R^+$ sufficiently small suppose initially that $(\lambda,z)\in B(0,\kappa)\subset\R^n\times Z$, $(\al,v)\in[-\ep,\ep]\times K$, $\al v\in B(0,\kappa)\subset X$. are such that $f(\lambda_\star+\lambda,\al v+z)=0$. Now we consider
		\begin{equation}
			0=f(\lambda_\star+\lambda,\al v+z)=\tsb{f(\lambda_\star+\lambda,\al v+z)-f(\lambda_\star+\lambda,\al v)}+\tsb{f(\lambda_\star+\lambda,\al v)-f(\lambda_\star+\lambda,0)}=\bf{I}+\bf{II}.
		\end{equation}
		We now write
		\begin{multline}
			\bf{I}=\tsb{f(\lambda_\star+\lambda,\al v+z)-f(\lambda_\star+\lambda,\al v)-D_2f(\lambda_\star+\lambda,\al v)z}\\+\tsb{D_2f(\lambda_\star+\lambda,\al v)-D_2f(\lambda_\star,0)}z+D_2f(\lambda_\star,0)z=\bf{I}_1+\bf{I}_2+D_2f(\lambda_\star,0)z
		\end{multline}
		and
		\begin{multline}
			\bf{II}=\tsb{f(\lambda_\star+\lambda,\al v)-f(\lambda_\star+\lambda,0)-D_2f(\lambda_\star+\lambda,0)\al v}\\+\tsb{D_2f(\lambda_\star+\lambda,0)-D_2f(\lambda_\star,0)-\lambda\cdot D_1D_2 f(\lambda_\star,0)}\al v+\sum_{j=1}^n\lambda_jD_{1,j}D_2f(\lambda_\star,0)\al v\\=\bf{II}_1+\bf{II}_2+\sum_{j=1}^n\lambda_jD_{1,j}D_2f(\lambda_\star,0)\al v.
		\end{multline}
		The terms $\bf{I}_1$, $\bf{I}_2$, $\bf{II}_1$ and $\bf{II}_2$ can be estimated as follows
		\begin{multline}\label{estimates on the remainder terms}
			\tnorm{\bf{I}_1}_Y\lesssim\tnorm{z}_X^2\lesssim\kappa\tnorm{z}_X,\;\tnorm{\bf{I}_2}_{Y}\lesssim\tp{|\lambda|+|\al v|}\tnorm{z}_X\lesssim\kappa\tnorm{z}_X,\\\tnorm{\bf{II}_1}_{Y}\lesssim\tnorm{\al v}^2_X\lesssim|\al|^2,\;\tnorm{\bf{II}_2}_{Y}\lesssim|\lambda|^2\tnorm{\al v}_{X}\lesssim\kappa|\al\lambda|,
		\end{multline}
		where the implicit constants depend on the second and third derivatives of the function $f$. Now we have that $0=\bf{I}+\bf{II}=D_2F(0,v,0,0)(\al\lambda,z)+\bf{I}_1+\bf{I}_2+\bf{II}_1+\bf{II}_2$ and so we may estimate according to the first claim~\eqref{the equivalence of lengths} that
		\begin{equation}\label{we are nearly there}
			|\al\lambda|+\tnorm{z}_X\lesssim\tnorm{D_2F(0,v,0,0)\tp{\al\lambda,z}}_Y=\tnorm{\bf{I}_1+\bf{I}_2+\bf{II}_1+\bf{II}_2}_Y\lesssim\kappa\tp{|\al\lambda|+\tnorm{z}_X}+\al^2.
		\end{equation}
		So by taking $\kappa>0$ sufficiently small we can absorb in~\eqref{we are nearly there} and acquire the desired bound~\eqref{the inequality of the second subclaim}. The open set $U_1$ can then be taken to be product of $B_{\R^n}(\lambda_\star,\kappa)$ and the preimage of $B(0,\kappa)\times B(0,\kappa)$ under the direct sum identification map $X=E_\star\oplus Z\to E_\star\times Z$. By taking $\kappa$ smaller (say $\kappa\le \ep$), if necessary, we can also ensure that
		\begin{equation}\label{inclusion of a ball!}
			\tcb{\al v\in E_\star\;:\;(\al,v)\in[-\ep,\ep]\times K}\supseteq B_{E_\star}(0,\kappa).
		\end{equation}
		The inclusion~\eqref{inclusion of a ball!} is where we are using the hypothesis $K\subseteq\tcb{v\in E_\star\;:\;\tnorm{v}_X=1}$. Finally, by making $\kappa$ smaller again, if needed, we can also guarantee that if $\al v\in B(0,\kappa)$ for some $(\al,v)\in[-\ep,\ep]\times K$, then $|\al|\le r/2c$, thanks to the fact that the set $K$ is a positive distance from $0$.
		
		We are now ready to prove the third item. It is clear that the inclusion `$\supseteq$' in~\eqref{totality of local solutions} holds, so we need only focus on justifying  `$\subseteq$'. So let us assume that $f(\lambda_\star+\lambda,x)=0$ with $(\lambda_\star+\lambda,x)\in U_1$ with $U_1$ as constructed in the proof of the second sub-claim. We decompose $x=w+z$ with $w\in E_\star$ and $z\in Z$ according to the direct sum $X=E_\star\oplus Z$. As inclusion~\eqref{inclusion of a ball!} holds, we may find $(\al,v)\in[-\ep,\ep]\times K$ such that $w=\al v$. We have now reached a situation in which the second sub-claim applies and we deduce that inequality~\eqref{the inequality of the second subclaim} holds for our $\lambda$ and $z$.
		
		If $\al=0$, then we deduce that $z=0$ and the solution belongs to the set of trivial solutions $U_1\cap\tcb{(\uplambda,0)\;:\;\uplambda\in\R}$. On the other hand, if $\al\neq 0$, we deduce the inequalities: $|\lambda|+\tnorm{z/\al}_{X}\le c|\al|\le r/2$. Hence we are in a position to apply the uniqueness assertion of the second item and deduce that $\lambda=\Lambda(\al,v)$ and $z/\al=W(\al,v)$.
	\end{proof}
	
	%*_*_*()(())()*_*_*()(())()*_*_*()(())()*_*_*()(())()*_*_*()(())()*_*_*()(())()*_*_*()(())()*_*_*()(())()
	\subsection{Analysis of the shallow water system's nonlinearities}\label{subsection on concrete nonlinear analysis}
	%*_*_*()(())()*_*_*()(())()*_*_*()(())()*_*_*()(())()*_*_*()(())()*_*_*()(())()*_*_*()(())()*_*_*()(())()
	
	We cast system~\eqref{traveling form of the equations} as a nonlinear operator equation with two parameters and then check that certain mapping properties are satisfied.
	
	Recall the definitions of the spaces $\Bar{\bf{X}}_L$, $\Bar{\bf{Y}}_L$ from~\eqref{larger than life spaces} and $\bf{X}_L$, $\bf{Y}_L$ from~\eqref{the domain spaces} and~\eqref{the codomain spaces}. For any $L\in\tp{\R^+}^2$ we define the map $N:\R^2\times\Bar{\bf{X}}_L\to\Bar{\bf{Y}}_L$ via
	\begin{equation}\label{definition of the nonlinearity}
		N(\gam,\kappa,u,\eta)=\bpm(1+\eta)u\cdot\grad u-\gam\eta\pd_1 u-\mu\grad\cdot(\eta\mathbb{S}u)+\eta\grad(1-\sig\Delta)\eta\\\grad\cdot(\eta u)\epm.
	\end{equation}
	
	\begin{prop}[Analysis of the nonlinearity]\label{proposition on analysis of the nonlinearity}
		For $L\in\tp{\R^+}^2$ the following hold.
		\begin{enumerate}
			\item The mapping $N$ from~\eqref{definition of the nonlinearity} is well-defined and smooth.
			\item $N$ is parity preserving in the sense that $N:\R^2\times\bf{X}_L\to\bf{Y}_L$.
			\item $N$ is purely nonlinear in the sense that $D_2N(\gam,\kappa,0,0)=0$ for all $(\gam,\kappa)\in\R^2$ where $D_2$ refers to the derivative with respect to the $\Bar{\bf{X}}_L$ factor of the domain.
			\item A tuple $(\gam,\kappa,u,\eta)\in\R^2\times\Bar{\bf{X}}_L$ is a solution to system~\eqref{traveling form of the equations} if and only if $(P+N)(\gam,\kappa,u,\eta)=0$, where $P$ is the family of linear maps defined in~\eqref{the family of linearized operators P}.
		\end{enumerate}
	\end{prop}
	\begin{proof}
		The mapping $N$ is the sum of bilinear and trilinear mappings. Therefore, to establish the first and second items, we need only verify that each constituent multilinearity of $N$ is a bounded mapping $\R^2\times \Bar{\bf{X}}_L\to \Bar{\bf{Y}}_L$ and that the defining symmetries of $\bf{X}_L$ and $\bf{Y}_L$ are respected. This latter point is rote checking, so let us only prove the former point. For the first component of $N$ we need only look the embedding maps $H^2\emb L^\infty$ and $H^3\emb W^{1,\infty}$ along with the elementary inequalities
		\begin{equation}
			\tnorm{(1+\eta)u\cdot\grad u}_{L_2}\le\tnorm{1+\eta}_{L^\infty}\tnorm{u}_{L^\infty}\tnorm{\grad u}_{L^2}\lesssim\tbr{\tnorm{\eta}_{H^3}}\tnorm{u}_{H^2}^2,
		\end{equation}
		\begin{equation}
			\tnorm{\eta\pd_1 u}_{L^2}\le\tnorm{\eta}_{L^\infty}\tnorm{\grad u}_{L^2}\lesssim\tnorm{\eta}_{H^3}\tnorm{u}_{H^2},
		\end{equation}
		\begin{equation}
			\tnorm{\grad\cdot\tp{\eta\mathbb{S}u}}_{L^2}\lesssim\tnorm{\eta\mathbb{S}u}_{H^1}\lesssim\tnorm{\eta}_{W^{1,\infty}}\tnorm{\mathbb{S}u}_{H^1}\lesssim\tnorm{\eta}_{H^3}\tnorm{u}_{H^2},
		\end{equation}
		and
		\begin{equation}
			\tnorm{\eta\grad\tp{1-\sig\Delta}\eta}_{L^2}\lesssim\tnorm{\eta}_{L^\infty}\tnorm{\eta}_{H^3}\lesssim\tnorm{\eta}_{H^3}^2.
		\end{equation}
		While for the second component of $N$ we shall use that $H^2(\R^2)$ is an algebra
		\begin{equation}
			\tnorm{\grad\cdot(\eta u)}_{H^1}\lesssim\tnorm{\eta u}_{H^2}\lesssim\tnorm{\eta}_{H^3}\tnorm{u}_{H^2}.
		\end{equation}
		So the first item is established. The third and fourth items are immediate.
	\end{proof}

	%*_*_*()(())()*_*_*()(())()*_*_*()(())()*_*_*()(())()*_*_*()(())()*_*_*()(())()*_*_*()(())()*_*_*()(())()
	\subsection{Conclusions}\label{subsection on conclusions}
	%*_*_*()(())()*_*_*()(())()*_*_*()(())()*_*_*()(())()*_*_*()(())()*_*_*()(())()*_*_*()(())()*_*_*()(())()
	
	In this final subsection of the document we bring together our linear analysis of Section~\ref{section on linear analysis} with the nonlinear analysis of Section~\ref{subsection on concrete nonlinear analysis} and with either the inverse function theorem or local bifurcation theorem of Section~\ref{theorem on multiparameter bifurcation}. We are studying solutions to the system of equations~\eqref{traveling form of the equations} which, thanks to the third item of Proposition~\ref{proposition on analysis of the nonlinearity}, are neatly encoded as zeros of the function $P+N$. Recall that the set $\mathfrak{E}\subset\tp{1,\infty}^2$ is defined in~\eqref{the set of interesting parameter values}.
	
	Our first result here, which is Theorem~\ref{3rd main theorem} of Section~\ref{subsection on main results and discussion}, regards the nonexistence of small amplitude periodic gravity-capillary roll wave solutions for certain combinations of wave speed and tilt.
	
	\begin{proof}[Proof of Theorem~\ref{3rd main theorem}]
		This is a simple consequence of the implicit function theorem. $P$ as defined in~\eqref{the family of linearized operators P} is a family of bounded linear maps while $N$ is smooth and satisfies $D_2N(\cdot,\cdot,0,0)=0$ thanks to Proposition~\ref{proposition on analysis of the nonlinearity}. Therefore, by Theorem~\ref{second thm on the operator P} we have that $D_2\tp{P+N}(\gam_\star,\kappa_\star,0,0)=P(\gam_\star,\kappa_\star,\cdot,\cdot)$ is an isomorphism $\Bar{\bf{X}}_L\to\Bar{\bf{Y}}_L$. So the implicit function theorem applies and so we learn that $\tp{P+N}\tp{\gam,\kappa,u,\eta}=0$ with $\gam-\gam_\star$, $\kappa-\kappa_\star$, $u$, $\eta$ sufficiently small implies that $u=0$ and $\eta=0$.
	\end{proof}
	
	Our next result, which is Theorem~\ref{1st main theorem} of Section~\ref{subsection on main results and discussion}, says that if the speed and tilt parameters are in the complement of the previous theorem's nonexistence set, then one can select the period lengths of the variables in such a way to guarantee the existence of a two parameter family of periodic nontrivial small amplitude solutions to system~\eqref{traveling form of the equations} which emanate from the trivial solution.
	
	\begin{proof}[Proof of Theorem~\ref{1st main theorem}]
		We shall prove this result by way of multiparameter bifurcation using Theorem~\ref{theorem on multiparameter bifurcation}. Let us first remark that it suffices to consider the case that $(\gam_\star,\kappa_\star)\in\mathfrak{E}$. Consider the reflection operator $R(y_1,y_2)=(-y_1,y_2)$; given $(u,\eta)\in\bf{X}_L$ let us define $(w,h)\in\bf{X}_L$ via $h=\eta\circ R$ and $w=Ru\circ R$; in other words
		\begin{equation}\label{reflection equation 1}
			h(x_1,x_2)=\eta(-x_1,x_2)\quad\text{and}\quad w(x_1,x_2)=(-u_1(-x_1,x_2),u_2(-x_1,x_2)),\quad (x_1,x_2)\in\T^2_L.
		\end{equation}
		A routine calculation shows that (for any choice of $\tp{\gam,\kappa}\in\R^2$)
		\begin{equation}\label{reflection equation 2}
			\tp{P+N}(\gam,\kappa,u,\eta)=0\quad\text{if and only if}\quad\tp{P+N}(-\gam,-\kappa,w,h)=0.
		\end{equation}
		Therefore, once the local theory of solutions for the parameter values $(\gam_\star,\kappa_\star)\in\mathfrak{E}$ has been established, we can deduce a corresponding result for $(-\gam_\star,-\kappa_\star)\in-\mathfrak{E}$ via this reflection procedure. So let us assume now that $(\gam_\star,\kappa_\star)\in\mathfrak{E}$.
		
		Our goal is to apply Theorem~\ref{theorem on multiparameter bifurcation} to the smooth mapping $f:\R^2\times\bf{X}_L\to\bf{Y}_L$ defined via $f=P+N$ at the point $(\gam_\star,\kappa_\star,0,0)$. The first hypothesis that $f(\cdot,\cdot,0,0)=0$ is satisfied by inspection. Next, we see that $D_2f(\gam_\star,\kappa_\star,0,0)=P(\gam_\star,\kappa_\star,\cdot,\cdot)$ (by the second item of Proposition~\ref{proposition on analysis of the nonlinearity}) and so we can apply Theorem~\ref{theorem on the operator P} and read off that the linear map $D_2f(\gam_\star,\kappa_\star,0,0):\bf{X}_L\to\bf{Y}_L$ has a two dimensional kernel $E=\m{ker}D_2f(\gam_\star,\kappa_\star,0,0)$, has a closed range $F=\m{ran}D_2f(\gam_\star,\kappa_\star,0,0)$ of codimension 2, and for all $(u,\eta)\in E\setminus\tcb{0}$ it holds that $\bf{Y}_L=F\oplus\m{span}D_{1,1}D_2f(\gam_\star,\kappa_\star,0,0)(u,\eta)\oplus\m{span}D_{1,2}D_2f(\gam_\star,\kappa_\star,0,0)(u,\eta)$. Hence all of the hypotheses of the multiparameter bifurcation theorem are satisfied.
		
		Let us take our complementing subspace: $Z=\bf{Z}_L(\gam,\kappa)\subset\bf{X}_L$ as in Theorem~\ref{theorem on kernel synthesis} and the compact set $K=E\cap\pd B_{\bf{X}_L}(0,1)$. Theorem~\ref{theorem on multiparameter bifurcation} grants us $\ep,r\in\R^+$ along with smooth functions $\Lambda:[-\ep,\ep]\times K\to B_{\R^2}(0,r)$, $W:[-\ep,\ep]\times K\to Z\cap B_{\bf{X}_L}(0,r)$ for which~\eqref{what are these functions really doing} is satisfied. To obtain the functions $\tp{\tilde{\gamma},\tilde{\kappa}}$ from $\Lambda$ and $\tp{\tilde{u},\tilde{\eta}}$ from $W$ we note that because $\Lambda$ and $W$ vanish for $\al=0$ we can divide by $\al$ and retain smoothness, i.e.:
		\begin{multline}
			\tp{\tilde{\gam},\tilde{\kappa}}(\al,\upupsilon,\upeta)=\Lambda(\al,\upupsilon,\upeta)/\al=\int_0^1D_1\Lambda(\tau\al,\upupsilon,\upeta)\;\m{d}\tau,\\\tp{\tilde{u},\tilde{\eta}}(\al,\upupsilon,\upeta)=W(\al,\upupsilon,\eta)/\al=\int_0^1D_1W(\tau\al,\upupsilon,\upeta)\;\m{d}\tau.
		\end{multline}
		We then take the parameter $\bar{r}\in\R^+$ to be large enough so that $(\tilde{\gam},\tilde{\kappa})$ and $(\tilde{u},\tilde{\eta})$ map into balls of radius $\bar{r}$.
		
		By construction we have $\tp{P+N}(\gam,\kappa,u,\eta)=0$ whenever the argument is defined via~\eqref{nontrivial solutions}. Thus to complete the proof of the first item, we need to argue that $\eta\neq0$ for $\al\neq0$. As $\bf{X}_L=E\oplus Z$ and $\al^2\tp{\tilde{u}(\al,\upupsilon,\eta),\tilde{\eta}\tp{\al,\upupsilon,\upeta}}\in Z$, we see that $\al(\upupsilon,\upeta)\neq 0$ implies $(u,\eta)\neq 0$. This happens exactly when $\al\neq0$ because $\tnorm{\tp{\upupsilon,\upeta}}_{\bf{X}_L}=1$. Moreover, since $(\upupsilon,\upeta)\in E=\m{ker}P(\gam_\star,\kappa_\star,\cdot,\cdot)$ we may quote Proposition~\ref{proposition on the reduction of kernels} to see that $\tp{\upupsilon,\upeta}=R_{\gam_\star,\kappa_\star}\upeta$ and hence $1\lesssim\tnorm{\upeta}_{\bf{X}_{L}^{\m{surf}}}$. By definition of $Z=\bf{Z}_L(\gam,\kappa)$ we also observe that $\tilde{\eta}(\al,\upupsilon,\upeta)$ and $\upeta\in\m{ker}Q(\gam_\star,\kappa_\star,\cdot)$ have disjoint Fourier supports and so in fact $\eta\neq0$ for $\al\neq0$.
		
		The validity of the second item of this theorem is immediate from the third conclusion of Theorem~\ref{theorem on multiparameter bifurcation} (which we are permitted to apply here as $K$ contains all kernel vectors of unit length) if we set $\mathcal{O}=U_1$. This completes the proof.
	\end{proof}
	
	Our final result, which is Theorem~\ref{2nd main thm} from Section~\ref{subsection on main results and discussion}, gives qualitative remarks about the solutions produced by the previous result.
	
	\begin{proof}[Proof of Theorem~\ref{2nd main thm}]
		Directly from equation~\eqref{nontrivial solutions} we know that $\eta_\al=\al\upeta+\al^2\tilde{\eta}(\al,\upupsilon,\upeta)$ and $u_\al=\al\upupsilon+\al^2\tilde{u}(\al,\upupsilon,\upeta)$ with $(\upupsilon,\upeta)\in K=\m{ker}P(\gam_\star,\kappa_\star,\cdot,\cdot)$. Hence, thanks to Proposition~\ref{proposition on the reduction of kernels}, we have that $\upeta\in\m{ker}Q(\gam_\star,\kappa_\star,\cdot)\setminus\tcb{0}$, where $Q$ is the operator from~\eqref{the family of auxilary linearized operators Q}, and
		\begin{equation}\label{the upupsilon in terms of the upeta}
			\upupsilon=\kappa_\star\tp{1-\gam\pd_1-\mu\Delta}^{-1}\tp{I+\mathcal{R}\otimes\mathcal{R}}\tp{\upeta e_1}-\gam_\star\mathcal{R}\otimes\mathcal{R}\tp{\upeta e_1}.
		\end{equation}
		Equation~\eqref{the upupsilon in terms of the upeta} readily implies that~\eqref{component 1} and~\eqref{component 2} hold. The kernel of $Q(\gam_\star,\kappa_\star,\cdot)$ is computed in Proposition~\ref{proposition on computation of the reduced kernel} and is spanned by the functions $\upvarphi_+$, $\upvarphi_-$ from equations~\eqref{border case} or~\eqref{interior case} and so $\upeta=A\upvarphi_++B\upvarphi_-$ for some $A,B\in\R$ (not both zero). So the desired estimates of~\eqref{deviation from first order approximation} follow by subtracting $\al\upeta$ from $\eta_\al$ and $\al\upupsilon$ from $u_\al$ and estimating the remainder in $L^\infty(\T^2_L)$ via the Sobolev embedding: $\tilde{\eta}(\al,\upupsilon,\upeta)\in H^3(\T^2_L)\emb L^\infty(\T^2_L)$, $\tilde{u}(\al,\upupsilon,\upeta)\in H^2(\T^2_L;\R^2)\emb L^\infty(\T^2_L;\R^2)$. The constant $c$ in the claimed inequalities can then be taken a function of $\bar{r}$ and these embedding constants.
	\end{proof}
	%*_*_*()(())()*_*_*()(())()*_*_*()(())()*_*_*()(())()*_*_*()(())()*_*_*()(())()*_*_*()(())()*_*_*()(())()
	
	%*_*_*()(())()*_*_*()(())()*_*_*()(())()*_*_*()(())()*_*_*()(())()*_*_*()(())()*_*_*()(())()*_*_*()(())()
	\bibliographystyle{abbrv}
	%*_*_*()(())()*_*_*()(())()*_*_*()(())()*_*_*()(())()*_*_*()(())()*_*_*()(())()*_*_*()(())()*_*_*()(())()
	\bibliography{bib.bib}

\begin{thebibliography}{10}

\bibitem{MR2259999}
N.~J. Balmforth and S.~Mandre.
\newblock Dynamics of roll waves.
\newblock {\em J. Fluid Mech.}, 514:1--33, 2004.

\bibitem{MR3600832}
B.~Barker, M.~A. Johnson, P.~Noble, L.~M. Rodrigues, and K.~Zumbrun.
\newblock Stability of viscous {S}t. {V}enant roll waves: from onset to infinite {F}roude number limit.
\newblock {\em J. Nonlinear Sci.}, 27(1):285--342, 2017.

\bibitem{MR2813829}
B.~Barker, M.~A. Johnson, L.~M. Rodrigues, and K.~Zumbrun.
\newblock Metastability of solitary roll wave solutions of the {S}t. {V}enant equations with viscosity.
\newblock {\em Phys. D}, 240(16):1289--1310, 2011.

\bibitem{MR1975092}
F.~Bouchut, A.~Mangeney-Castelnau, B.~Perthame, and J.-P. Vilotte.
\newblock A new model of {S}aint {V}enant and {S}avage-{H}utter type for gravity driven shallow water flows.
\newblock {\em C. R. Math. Acad. Sci. Paris}, 336(6):531--536, 2003.

\bibitem{MR2118849}
F.~Bouchut and M.~Westdickenberg.
\newblock Gravity driven shallow water models for arbitrary topography.
\newblock {\em Commun. Math. Sci.}, 2(3):359--389, 2004.

\bibitem{MR2397996}
M.~Boutounet, L.~Chupin, P.~Noble, and J.~P. Vila.
\newblock Shallow water viscous flows for arbitrary topography.
\newblock {\em Commun. Math. Sci.}, 6(1):29--55, 2008.

\bibitem{MR2562163}
D.~Bresch.
\newblock Shallow-water equations and related topics.
\newblock In {\em Handbook of differential equations: evolutionary equations. {V}ol. {V}}, Handb. Differ. Equ., pages 1--104. Elsevier/North-Holland, Amsterdam, 2009.

\bibitem{BROCK69}
R.~R. Brock.
\newblock Development of roll-wave trains in open channels.
\newblock {\em Journal of the Hydraulics Division}, 95(4):1401--1427, 1969.

\bibitem{brownfield2023slowlytravelinggravitywaves}
J.~Brownfield and H.~Q. Nguyen.
\newblock Slowly traveling gravity waves for darcy flow: existence and stability of large waves, 2023.
\newblock Preprint, 2023. \href{https://arxiv.org/abs/2312.07446}{arXiv:2312.07446}.

\bibitem{MR1780437}
H.-C. Chang, E.~A. Demekhin, and E.~Kalaidin.
\newblock Coherent structures, self-similarity, and universal roll wave coarsening dynamics.
\newblock {\em Phys. Fluids}, 12(9):2268--2278, 2000.

\bibitem{MR0288640}
M.~G. Crandall and P.~H. Rabinowitz.
\newblock Bifurcation from simple eigenvalues.
\newblock {\em J. Functional Analysis}, 8:321--340, 1971.

\bibitem{MR0033717}
R.~F. Dressler.
\newblock Mathematical solution of the problem of roll-waves in inclined open channels.
\newblock {\em Comm. Pure Appl. Math.}, 2:149--194, 1949.

\bibitem{MR1821555}
J.-F. Gerbeau and B.~Perthame.
\newblock Derivation of viscous {S}aint-{V}enant system for laminar shallow water; numerical validation.
\newblock {\em Discrete Contin. Dyn. Syst. Ser. B}, 1(1):89--102, 2001.

\bibitem{Groves_2004}
M.~D. Groves.
\newblock {Steady water waves}.
\newblock {\em J. Nonlinear Math. Phys.}, 11(4):435--460, 2004.

\bibitem{MR4406719}
S.~V. Haziot, V.~M. Hur, W.~A. Strauss, J.~F. Toland, E.~Wahl\'{e}n, S.~Walsh, and M.~H. Wheeler.
\newblock Traveling water waves---the ebb and flow of two centuries.
\newblock {\em Quart. Appl. Math.}, 80(2):317--401, 2022.

\bibitem{MR0890282}
S.~H. Hwang and H.-C. Chang.
\newblock Turbulent and inertial roll waves in inclined film flow.
\newblock {\em Phys. Fluids}, 30(5):1259--1268, 1987.

\bibitem{JEFF25}
H.~Jeffreys.
\newblock {LXXXIV}. {T}he flow of water in an inclined channel of rectangular section.
\newblock {\em The London, Edinburgh, and Dublin Philosophical Magazine and Journal of Science}, 49(293):793--807, 1925.

\bibitem{MR3933411}
M.~A. Johnson, P.~Noble, L.~M. Rodrigues, Z.~Yang, and K.~Zumbrun.
\newblock Spectral stability of inviscid roll waves.
\newblock {\em Comm. Math. Phys.}, 367(1):265--316, 2019.

\bibitem{MR3219516}
M.~A. Johnson, P.~Noble, L.~M. Rodrigues, and K.~Zumbrun.
\newblock Behavior of periodic solutions of viscous conservation laws under localized and nonlocalized perturbations.
\newblock {\em Invent. Math.}, 197(1):115--213, 2014.

\bibitem{MR2784868}
M.~A. Johnson, K.~Zumbrun, and P.~Noble.
\newblock Nonlinear stability of viscous roll waves.
\newblock {\em SIAM J. Math. Anal.}, 43(2):577--611, 2011.

\bibitem{MR2859263}
H.~Kielh\"ofer.
\newblock {\em Bifurcation theory}, volume 156 of {\em Applied Mathematical Sciences}.
\newblock Springer, New York, second edition, 2012.
\newblock An introduction with applications to partial differential equations.

\bibitem{koganemaru2023travelingwavesolutionsfree}
J.~Koganemaru and I.~Tice.
\newblock Traveling wave solutions to the free boundary incompressible navier-stokes equations with navier boundary conditions.
\newblock Preprint, 2023. \href{https://arxiv.org/abs/2311.01590}{arXiv:2311.01590}.

\bibitem{koganemaru2022traveling}
J.~Koganemaru and I.~Tice.
\newblock Traveling wave solutions to the inclined or periodic free boundary incompressible {N}avier-{S}tokes equations.
\newblock {\em J. Funct. Anal.}, 285(7):Paper No. 110057, 75, 2023.

\bibitem{MR1194982}
C.~Kranenburg.
\newblock On the evolution of roll waves.
\newblock {\em J. Fluid Mech.}, 245:249--261, 1992.

\bibitem{MR4105349}
D.~Lannes.
\newblock Modeling shallow water waves.
\newblock {\em Nonlinearity}, 33(5):R1--R57, 2020.

\bibitem{leoni2019traveling}
G.~Leoni and I.~Tice.
\newblock Traveling wave solutions to the free boundary incompressible {N}avier-{S}tokes equations.
\newblock {\em Comm. Pure Appl. Math.}, 76(10):2474--2576, 2023.

\bibitem{MR1324142}
R.~Liska, L.~Margolin, and B.~Wendroff.
\newblock Nonhydrostatic two-layer models of incompressible flow.
\newblock {\em Comput. Math. Appl.}, 29(9):25--37, 1995.

\bibitem{MR2281291}
F.~Marche.
\newblock Derivation of a new two-dimensional viscous shallow water model with varying topography, bottom friction and capillary effects.
\newblock {\em Eur. J. Mech. B Fluids}, 26(1):49--63, 2007.

\bibitem{mascia_2010}
C.~Mascia.
\newblock A dive into shallow water.
\newblock {\em Riv. Math. Univ. Parma (N.S.)}, 1(1):77--149, 2010.

\bibitem{MR0853213}
J.~H. Merkin and D.~J. Needham.
\newblock On infinite period bifurcations with an application to roll waves.
\newblock {\em Acta Mech.}, 60(1-2):1--16, 1986.

\bibitem{1984RSPSA.394..259N}
D.~J. {Needham} and J.~H. {Merkin}.
\newblock {On Roll Waves down an Open Inclined Channel}.
\newblock {\em Proceedings of the Royal Society of London Series A}, 394(1807):259--278, Aug. 1984.

\bibitem{Ng_Mei_1994}
C.-O. Ng and C.~C. Mei.
\newblock Roll waves on a shallow layer of mud modelled as a power-law fluid.
\newblock {\em Journal of Fluid Mechanics}, 263:151–184, 1994.

\bibitem{nguyen2023largetravelingcapillarygravitywaves}
H.~Q. Nguyen.
\newblock Large traveling capillary-gravity waves for darcy flow.
\newblock Preprint, 2023. \href{https://arxiv.org/abs/2311.01299}{arXiv:2311.01299}.

\bibitem{nguyen_tice_2022}
H.~Q. Nguyen and I.~Tice.
\newblock Traveling wave solutions to the one-phase {M}uskat problem: existence and stability.
\newblock {\em Arch. Ration. Mech. Anal.}, 248(1):Paper No. 5, 58, 2024.

\bibitem{MR2338351}
P.~Noble.
\newblock Existence of pulsating roll-waves for the {S}aint {V}enant system.
\newblock {\em Arch. Ration. Mech. Anal.}, 186(1):53--76, 2007.

\bibitem{PCC91}
T.~Prokopiou, M.~Cheng, and H.-C. Chang.
\newblock Long waves on inclined films at high reynolds number.
\newblock {\em Journal of Fluid Mechanics}, 222:665–691, 1991.

\bibitem{MR3449905}
L.~M. Rodrigues and K.~Zumbrun.
\newblock Periodic-coefficient damping estimates, and stability of large-amplitude roll waves in inclined thin film flow.
\newblock {\em SIAM J. Math. Anal.}, 48(1):268--280, 2016.

\bibitem{stevenson2023shallow}
N.~Stevenson and I.~Tice.
\newblock The traveling wave problem for the shallow water equations: well-posedness and the limits of vanishing viscosity and surface tension.
\newblock Preprint, 2023. \href{https://arxiv.org/abs/2311.00160}{arXiv:2311.00160}.

\bibitem{stevensontice2023wellposedness}
N.~Stevenson and I.~Tice.
\newblock Well-posedness of the stationary and slowly traveling wave problems for the free boundary incompressible navier-stokes equations.
\newblock Preprint, 2023. \href{https://arxiv.org/abs/2306.15571}{arXiv:2306.15571}.

\bibitem{stevenson2023wellposedness}
N.~Stevenson and I.~Tice.
\newblock Well-posedness of the traveling wave problem for the free boundary compressible {N}avier-{S}tokes equations.
\newblock Preprint, 2023. \href{https://arxiv.org/abs/2301.00773}{arXiv:2301.00773}.

\bibitem{MR4337506}
N.~Stevenson and I.~Tice.
\newblock Traveling wave solutions to the multilayer free boundary incompressible {N}avier-{S}tokes equations.
\newblock {\em SIAM J. Math. Anal.}, 53(6):6370--6423, 2021.

\bibitem{Strauss_2010}
W.~A. Strauss.
\newblock {Steady water waves}.
\newblock {\em Bull. Amer. Math. Soc. (N.S.)}, 47(4):671--694, 2010.

\bibitem{Toland_1996}
J.~F. Toland.
\newblock {Stokes waves}.
\newblock {\em Topol. Methods Nonlinear Anal.}, 7(1):1--48, 1996.

\bibitem{MR0570367}
H.~Tougou.
\newblock Stability of turbulent roll-waves in an inclined open channel.
\newblock {\em J. Phys. Soc. Japan}, 48(3):1018--1023, 1980.

\bibitem{yangZ2023}
Z.~Yang and K.~Zumbrun.
\newblock Multidimensional stability and transverse bifurcation of hydraulic shocks and roll waves in open channel flow, 2023.
\newblock Preprint, 2023. \href{https://arxiv.org/abs/2309.08870}{arXiv:2309.08870}.

\bibitem{MR1187443}
J.~Yu and J.~Kevorkian.
\newblock Nonlinear evolution of small disturbances into roll waves in an inclined open channel.
\newblock {\em J. Fluid Mech.}, 243:575--594, 1992.

\end{thebibliography}
	%*_*_*()(())()*_*_*()(())()*_*_*()(())()*_*_*()(())()*_*_*()(())()*_*_*()(())()*_*_*()(())()*_*_*()(())()
\end{document}